\documentclass[reqno,12pt]{amsart}
\def\vint_#1{\mathchoice%
	{\mathop{\kern 0.2em\vrule width 0.6em height 0.69678ex depth -0.58065ex
			\kern -0.8em \intop}\nolimits_{\kern -0.4em#1}}%
	{\mathop{\kern 0.1em\vrule width 0.5em height 0.69678ex depth -0.60387ex
			\kern -0.6em \intop}\nolimits_{#1}}%
	{\mathop{\kern 0.1em\vrule width 0.5em height 0.69678ex
			depth -0.60387ex
			\kern -0.6em \intop}\nolimits_{#1}}%
	{\mathop{\kern 0.1em\vrule width 0.5em height 0.69678ex depth -0.60387ex
			\kern -0.6em \intop}\nolimits_{#1}}}
\def\vintslides_#1{\mathchoice% 
	{\mathop{\kern 0.1em\vrule width 0.5em height 0.697ex depth -0.581ex
			\kern -0.6em \intop}\nolimits_{\kern -0.4em#1}}%
	{\mathop{\kern 0.1em\vrule width 0.3em height 0.697ex depth -0.604ex
			\kern -0.4em \intop}\nolimits_{#1}}%
	{\mathop{\kern 0.1em\vrule width 0.3em height 0.697ex depth -0.604ex
			\kern -0.4em \intop}\nolimits_{#1}}%
	{\mathop{\kern 0.1em\vrule width 0.3em height 0.697ex depth -0.604ex
			\kern -0.4em \intop}\nolimits_{#1}}}

\usepackage{a4wide}
\usepackage{mathtools}
\usepackage{amsthm}
\usepackage{amssymb}
\usepackage{hyperref}

\usepackage[capitalise]{cleveref}
\usepackage{glossaries}
\expandafter\def\csname ver@etex.sty\endcsname{3000/12/31}

\usepackage{autonum}
\usepackage{color}
\usepackage{enumerate}
\usepackage[obeyFinal]{todonotes}

\usepackage{csquotes}

\newcommand{\R}{\mathbb{R}}

\newcommand{\N}{\mathbb{N}}

\newcommand{\m}{\mathfrak{m}}
\newcommand{\n}{\mathfrak{n}}

\renewcommand{\d}{\,\mathrm{d}}

\renewcommand{\liminf}{\varliminf}
\renewcommand{\limsup}{\varlimsup}
\newcommand{\Ric}{\mathrm{Ric}}

\newcommand{\loc}{\mathrm{loc}}
\newcommand{\lip}{\mathrm{lip}}
\newcommand{\cd}{\mathrm{CD}}
\newcommand{\rcd}{\mathrm{RCD}}
\newcommand{\ent}{\mathrm{Ent}}
\newcommand{\ch}{\mathrm{Ch}}

\newcommand{\p}{\mathcal{P}}
\newcommand{\spt}{\mathrm{spt}\,}

\newcommand{\Lip}{\mathrm{Lip}}
\newcommand{\Lipb}{\mathrm{Lip}_{\mathrm{b}}}

\numberwithin{equation}{section}
\newtheorem{theorem}{Theorem}[section]
\newtheorem{proposition}[theorem]{Proposition}
\newtheorem{lemma}[theorem]{Lemma}

\newtheorem{corollary}[theorem]{Corollary}
\theoremstyle{definition}
\newtheorem{definition}[theorem]{Definition}
\newtheorem{assumption}[theorem]{Assumption}
\theoremstyle{remark}

\newtheorem{remark}[theorem]{Remark}

\newcommand{\aint}[2][]{%                           %Integral average
	\ifthenelse{\equal{#1}{}}%
	{%
		\mathchoice%
		{\mathop{\kern 0.2em\vrule width 0.6em height 0.69678ex depth -0.58065ex
				\kern -0.8em \intop}\nolimits_{\kern -0.45em#2}^{#1}}%
		{\mathop{\kern 0.1em\vrule width 0.5em height 0.69678ex depth -0.60387ex
				\kern -0.6em \intop}\nolimits_{#2}^{#1}}%
		{\mathop{\kern 0.1em\vrule width 0.5em height 0.69678ex depth -0.60387ex
				\kern -0.6em \intop}\nolimits_{#2}^{#1}}%
		{\mathop{\kern 0.1em\vrule width 0.5em height 0.69678ex depth -0.60387ex
				\kern -0.6em \intop}\nolimits_{#2}^{#1}}}%
	{%
		\mathchoice%
		{\mathop{\kern 0.2em\vrule width 0.6em height 0.69678ex depth -0.58065ex
				\kern -0.8em \intop}\nolimits_{\kern -0.45em#1}^{#2}}%
		{\mathop{\kern 0.1em\vrule width 0.5em height 0.69678ex depth -0.60387ex
				\kern -0.6em \intop}\nolimits_{#1}^{#2}}%
		{\mathop{\kern 0.1em\vrule width 0.5em height 0.69678ex depth -0.60387ex
				\kern -0.6em \intop}\nolimits_{#1}^{#2}}%
		{\mathop{\kern 0.1em\vrule width 0.5em height 0.69678ex depth -0.60387ex
				\kern -0.6em \intop}\nolimits_{#1}^{#2}}}} 
\begin{document}
	
	\parindent=0in
	\title[On the rigidity of Wasserstein contraction along heat flows]
	{On the rigidity of Wasserstein contraction along heat flows}

	\author{Zhenhao Li}

	\address{Faculty of Mathematics \\
		Bielefeld University \\
		Postfach 10 01 31 \\
		33501 Bielefeld \\
		Germany.}
	\email{zhenhao.li@math.uni-bielefeld.de}
	\begin{abstract}
	We establish an equivalence between the rigidity of Wasserstein contraction along heat flows and the rigidity of Bakry--Émery gradient estimates for Lipschitz functions.  
	Applying results of Ambrosio--Brué--Semola and Han, we show that if an $\rcd$ space with Ricci lower bound $K\in[0,\infty)$ admits two distinct points $x,y$ such that the $2$-Wasserstein distance between the associated heat kernels satisfies
	\[
	W_2(p_t(x,\cdot), p_t(y,\cdot)) = e^{-Kt} d(x,y),
	\]
	then the space splits off a line.
	
	Moreover, for weighted smooth manifolds, we provide a direct proof of the rigidity theorem for all curvature bounds $K \in \mathbb{R}$.
	In particular, we characterize a class of weighted Euclidean spaces as the only spaces where the Wasserstein contraction is sharp for all pairs of points.
	\end{abstract}
\maketitle
\section{Introduction}
Since a synthetic Ricci curvature lower bound for metric measure spaces (mms for short) was introduced by Sturm in \cite{sturm2006--1,Sturm06-2}, and Lott and Villani in \cite{LottVillani09}, various equivalent formulations have been proposed and extensively studied.
Among these, one key characterization expresses the Ricci curvature lower bound via an (exponential) contraction property of the heat flow w.r.t. the Wasserstein distance.
More precisely, a mms (under suitable assumptions) has Ricci curvature bounded below by $K$ if for all $x,y$ and all $t>0$, the $2$-Wasserstein distance between heat kernels $p_t(x,\cdot)$ and $p_t(y,\cdot)$ satisfies
\begin{equation}\label{eq:WassersteinK}
	W_2(p_t(x,\cdot),p_t(y,\cdot))\leq e^{-Kt}d(x,y).
\end{equation}
This estimate is sharp, for instance, in Euclidean space where the Wasserstein distance between two heat flows starting from Dirac measures remains constant in time.

A natural question is: for which spaces with Ricci lower bound by $K$, do there exist points \( x \neq y \) and time \( t > 0 \) such that equality holds in \eqref{eq:WassersteinK}?

In this work, we answer this question by establishing the following rigidity results:
\begin{theorem}\label{thm:1.1}
	Let $(X,d,\m)$ be one of the following
	\begin{enumerate}
		\item an $\rcd(K,N)$ mms for $K\geq 0$, $N\in[1,\infty]$ with either $K>0$ or $N<\infty$;
		\item an $\rcd(K,\infty)$ weighted manifold with $K\in\R$ where the heat flow is ultracontractive.
	\end{enumerate}
	If equality in \eqref{eq:WassersteinK} is attained by a pair of distinct points $x,y\in X$, then the space splits off a line; that is, there exists a mms $(Y,d_Y,\m_Y)$, either satisfying $\rcd(K,\infty)$ or being a singleton, s.t. $(X,d,\m)$ is isomorphic as a mms to the product space
	\begin{equation}
		(\R^1,|\cdot|,e^{-f}\mathcal{L}^1)\times (Y,d_Y,\m_Y),\quad \text{where $f''\equiv K$ on $\R^1$. }
	\end{equation}
\end{theorem}
The central new ingredient in the proof of the first part of \cref{thm:1.1} is an equivalence between the attainment of equality in Wasserstein contraction and the attainment of equality in the dual gradient estimates for Lipschitz functions.
This allows us to leverage results by Ambrosio--Bru\'e--Semola \cite{ABS19Gafa} and Han \cite{Han-JMPA21} to conclude the splitting of the space.

For the second part, we give a direct proof of the rigidity result that also covers $K<0$, by analysing the behaviour of the relative entropy along Wasserstein geodesics.
Then further combining with a dimensional refinement of displacement convexity from \cite{EKS15}, we obtain an additional rigidity on the curvature bound.
\begin{theorem}
Let $(X,d,\m)$ be a finite-dimensional $\rcd(K,\infty)$ mms for $K\in\R$ s.t.  either $K\geq 0 $ or $(X,d,\m)$ satisfies $\rcd(K,N)$ for some $N<\infty$.
If equality in \eqref{eq:WassersteinK} is attained by a pair of distinct points $x,y\in X$, then $K=0$ and the space splits off a line.
\end{theorem}
\iffalse
from both Eulerian and Lagrangian perspectives:
\begin{itemize}
	\item by establishing a duality on the rigidity of Wasserstein contraction and the rigidity of Bakry-\'Emery gradient estimates, we apply the results of Ambrosio--Bru\'e--Semola \cite{ABS19Gafa} for $K=0$ and Han \cite{Han-JMPA21} for $K>0$ to show an $\rcd$ space splits off a line if the Wasserstein contraction is sharp at some pair of points;
	\item by analysing the behaviour of the Boltzmann entropy along Wasserstein geodesics on which the contraction is sharp, we give a direct proof of the rigidity theorem for weighted manifolds.
\end{itemize} 
\fi
\subsection{Duality on Wasserstein contractions and gradient estimates}
The notion of $\rcd(K,\infty)$ mms was introduced in \cite{AGSRCD2014,Gigli15MAMS}. 
On an $\rcd$ space $(X,d,\m)$, the heat flow, viewed as the $L^2$-gradient flow of Cheeger's energy, generates a linear semigroup $(P_t)_{t\geq 0}$ on $L^2(X,\m)$.
On the other hand, for any $x\in X$ it also induces a Wasserstein gradient flow of probability measures $(p_t(x,\d\m))_{t>0}$ for the relative entropy, starting at the Dirac measure $\delta_x$.

For each pair of H\"older exponents $p,q\in[1,\infty]$, the following two properties are equivalent:
\begin{itemize}
\item \textbf{$p$-Wasserstein contraction}: for all $x,y\in X$ and $t>0$, it holds that
\begin{equation}\label{introineq:C_p}
	W_p(p_t(x,\d\m),p_t(y,\d\m))\leq e^{-Kt}d(x,y);
\end{equation}
\item \textbf{$q$-Bakry--Émery gradient estimate}: for every bounded Lipschitz function $f$, the following holds $\m$-a.e.
\begin{align}
		|\nabla P_t f| &\leq  e^{-Kt}\big(P_t(|\nabla f|^q)\big)^{1/q}, && q\in[1,\infty)\label{introineq:G_q}\\
			|\nabla P_t f|& \leq e^{-Kt}\|\nabla f\|_{L^\infty},&& q=\infty\label{introineq:G_infty}.
\end{align}
\end{itemize}
This duality was established by Kuwada \cite{Kuwadaduality} for the local Lipschitz constant, later by Ambrosio, Gigli and Savar\'e in \cite{AGSRCD2014,AGS15} for the minimal relaxed gradient and $p=2$, and finally extended to all $p\in[1,\infty]$ by Savar\'e in \cite{Savare14}.

The rigidity problem of the gradient estimates was first addressed by Ambrosio, Bru\'e and Semola \cite{ABS19Gafa}.
Using the theory of regular Lagrangian flows and Gigli’s splitting theorem \cite{Giglisplitting13,Gigli14splitting}, they showed that an $\rcd(0,N)$ space must split off a line if there exists a non-constant bounded Lipschitz function attaining equality in the $1$-gradient estimate.
This strategy was later adapted by Han in \cite{Han-JMPA21} to treat the case of $\rcd(K,\infty)$ spaces with $K>0$.

One of the main contributions of this work is to establish an equivalence between the attainment of equality the gradient estimates and in the Wasserstein contractions.
\begin{theorem}
	Let $(X,d,\m)$ be a proper $\rcd(K,\infty)$ mms for $K\in\R$. Then 
	\begin{enumerate}
		\item if there exist $x \neq y \in X$ such that equality in \eqref{introineq:C_p} holds for $p = 2$, equality in \eqref{introineq:G_infty} is also achieved by some non-constant Lipschitz function;
		\item conversely, if there exists a non-constant Lipschitz function for which \eqref{introineq:G_infty} holds with equality, then for all $p \in [1,\infty)$, one can find $x \neq y \in X$ such that equality is attained in \eqref{introineq:C_p}, provided the heat flow is ultracontractive.
	\end{enumerate}
Moreover, on a finite-dimensional $\operatorname{RCD}(K,\infty)$ space, the sharp $p$-Wasserstein contraction is attained if and only if the sharp $q$-gradient estimate is attained, for all $p \in [1,\infty)$ and $q \in [1,\infty]$.
\end{theorem}
We refer to \cref{sec:dualitythm} for a more detailed statement.
\smallskip

The strategy of the proof is as follows. 
Assuming that equality in \eqref{introineq:C_p} is attained, we first show that the Kantorovich potential for the optimal transport between the given heat kernels is Lipschitz—this follows from the slope estimate for Kantorovich potentials in \cite{AGSInvention} and Savar\'e’s $\infty$-Wasserstein contraction result in \cite{Savare14}. 
Then, adapting Kuwada’s argument in his duality theorem, we prove that equality in the Wasserstein contraction implies equality in \eqref{introineq:G_infty} for the Kantorovich potential.

For the reverse implication, we consider a Lipschitz function that attains equality in the gradient estimate. 
Along its gradient flow curve, we show that equality in \eqref{introineq:C_p} for $p=1$ is attained, using this function as a candidate Kantorovich potential. 
The rigorous construction of gradient flow curves in the non-smooth setting relies on the theory of regular Lagrangian flows, along with a suitable truncation argument.
\smallskip

We would like to point out one technical issue here, that is not present in \cite{ABS19Gafa}: functions considered in our setting (e.g. the ones attaining equality in \eqref{introineq:G_infty} or Kantorovich potentials for optimal transports) may not be bounded or integrable.
Therefore, it is necessary to extend the gradient estimates to all (possibly unbounded) Lipschitz functions, which will be discussed in detail in \cref{sec:gradientestimate}.

Moreover, it seems unnatural to impose a finite-dimensionality assumption on the underlying space.
In fact, the author conjectures unless $K=0$ that no finite dimensional $\rcd(K,\infty)$ space attains sharp Wasserstein contraction.
Thus additional technical challenges arise from the absence of the doubling property of the reference measure and the lack of Gaussian heat kernel bounds.
These are overcome with the help of the exponential integrability of heat kernels on $\rcd(K,\infty)$ spaces shown by Tamanini \cite{Tamanini-19HK} and several dimension-free functional inequalities; see \cref{sec:heatkernel} for more details.
\subsection{On weighted smooth manifolds}
Given a smooth Riemannian manifold \( (M, g) \) equipped with the weighted volume measure \( \m\coloneqq e^{-f} \mathrm{vol}_g \) where $f$ is a smooth function, the associated weighted Ricci curvature (also known as the Bakry--Émery curvature) is defined as
\begin{equation}
	\Ric_\m \coloneqq \Ric_g + \nabla^2 f.
\end{equation}
An illustrative example is the weighted Euclidean line
\begin{equation}\label{eq:introR1k}
 (\R^1,|\cdot|,e^{-f }\mathcal{L}^1), \quad \text{with }f''\equiv K
\end{equation}
which satisfies $\rcd(K,\infty)$.
One can check that in this setting the Lipschitz function $\varphi(x)\equiv x$ attains the sharp $\infty$-gradient estimate.
In particular, equality in \eqref{introineq:C_p} is achieved for every pair of points.

In \cref{sec:proofsharpness}, we show that a weighted smooth manifold $(M,g,\m)$ with $\Ric_\m\geq K$ attains sharp Wasserstein contraction at two distinct points if and only if it splits with a one-dimensional factor given as \eqref{eq:introR1k}; see \cref{thm:sharpness}.
Moreover, iterating the splitting we characterize the Euclidean spaces $\R^n$ with weight function $f$ satisfying $\nabla^2 f\equiv K\mathrm{id}$ as the only spaces in which the sharp Wasserstein contraction is attained between every pair of points; see \cref{thm:rigidity}.

The proof follows the spirit of Otto's calculus on the Wasserstein space.
Assuming the equality in \eqref{introineq:C_p} holds at $x\neq y\in M$ for $p=2$, we consider the geodesic $(\gamma^a)_{a\in[0,1]}$ joining them.
We show that the curve of heat flows $p_t(\gamma^a,\d\m)$ is a geodesic in the $2$-Wasserstein space, along which the relative entropy is a quadratic function of $a$, hence has constant second derivative.

We show that one can pass the second derivative of entropy along the Wasserstein geodesic to the Kantorovich potential, and deduce that it must have vanishing Hessian. 
This, in turn, implies that the underlying space splits.
\iffalse

Finally, using the constancy of the second derivative of entropy along the Wasserstein geodesic and the dimensional refinement of displacement convexity from \cite{EKS15}, we obtain further rigidity on the curvature bound (see \cref{thm:Kmustbe0})
\begin{itemize}
	\item if an $\rcd(K,N)$ mms with $K\in\R$ and $N<\infty$ attains sharp Wasserstein contraction or sharp gradient estimate, then necessarily $K=0$;
	\item if a finite-dimensional $\rcd(K,\infty)$ mms with $K\geq 0$ attains sharp Wasserstein contraction or a sharp gradient estimate, then again $K=0$.
\end{itemize} 
\fi
\iffalse
We will prove the above statements in \cref{sec:proofsharpness}.
Proceeding to that, we present in \cref{sec:prep} some prerequisites on splitting theorems, optimal transport, and heat kernel estimates.
In \cref{sec:furtherrem}, we provide additional remarks, with a particular focus on the related work of Ambrosio, Brué, and Semola \cite{ABS19Gafa} concerning the rigidity of the $1$-Bakry--Émery gradient estimate.
\fi
\medskip

\noindent \textbf{Acknowledgement.}
The author is grateful for the comments of an anonymous referee for drawing his attention to the work of Ambrosio, Brué, and Semola \cite{ABS19Gafa}.
He also thanks Matthias Erbar for valuable discussions and Timo Schultz for reading an earlier draft of the paper.
\section{Preliminaries}\label{sec:prep}
\subsection*{Notations and conventions}
A metric measure space $(X,d,\m)$ consists of a Polish metric space $(X,d)$ and a locally finite Borel measure $\m$ with \textbf{full support}.
A triple $(M,g,\m)$ always stands for a weighted complete smooth Riemannian manifold with the reference measure $\m\coloneqq e^{-f}\mathrm{vol}_g$ for some smooth function $f\colon M\to\R$.

The notation $A(\cdot)\lesssim_{a,b}B(\cdot)$ means that there exists a constant $C$ depending only on parameters $a,b$ such that $A(\cdot)\leq C\cdot B(\cdot)$.
\subsection{Wasserstein geodesics and Kantorovich potentials}
We assume that the reader is familiar with basic notions on optimal transport.
Otherwise, we refer to \cite{AmbrosioGigli2013} for an approachable exposition.

Let $(X,d,\m)$ be a metric measure space. 
For $p\in(1,\infty)$, a $p$-\emph{Wasserstein geodesic} always stands for a constant speed geodesic on the space $\p_p(X)$ of all probability measures with finite $p$-th moment w.r.t. the Wasserstein distance $W_p$.
Given $\mu,\nu\in \p_p(X)$, a function $\phi$ is called a \emph{Kantorovich potential} from $\mu$ to $\nu$ (w.r.t. $W_p$)\footnote{We often omit explicit mention of the exponent when there is no risk of ambiguity. This applies also to terms such as Wasserstein geodesic or Hopf--Lax semigroup.} if 
\begin{equation}\label{eq:Kduality2}
\frac1pW_p^p(\mu,\nu)=\int \phi\d\mu+\int \phi^c\d \nu,\quad \phi^c(y)\coloneqq \inf_{x\in X}\left(\frac1pd^p(x,y)-\phi(x)\right).
\end{equation}
Moreover, set $L(s)\coloneqq p^{-1}s^p$ and its Legendre conjugate $L^*$ as $L^*(s)\coloneqq q^{-1}s^q$ with $q$ the H\"older dual of $p$.
Define the Hopf--Lax semigroup of the Kantorovich potential $\phi$ by 
\begin{equation}\label{eq:HLsemigroup1}
	Q_a \phi(x)\coloneqq \inf_{y\in X}\left[  \phi(y)+aL\left(\frac{d(x,y)}{a}\right) \right].
\end{equation}
We can formulate the conjugate $f^c$ by $Q_1(-f)$ and define the \emph{intermediate-time Kantorovich potentials} for all $a\in[0,1)$ by 
	\begin{equation}\label{eq:IntKantorovich}
	\phi^a(x)\coloneqq Q_{1-a}(-\phi^c)(x)= \inf_{y\in X}\left( \frac{d^p(x,y)}{p(1-a)^{p-1}}-\phi^c(y)\right),
\end{equation}
and $\phi^1\coloneqq -\phi^c$. 

We recall the following interpolation of Kantorovich potentials at intermediate times along Wasserstein geodesics. The proof can be found in \cite{AmbrosioGigli2013} for $p=2$, \cite{Cavalletti-McCann21} for general $p\in(1,\infty)$, and see also \cite[Chapter 7]{Villani}.
\begin{lemma}\label{lemma:interPotential}
	Let $(\mu^a)_{a\in[0,1]}$ be a $p$-Wasserstein geodesic and let $(\phi)_{a\in[0,1]}$ be an associated family of intermediate-time Kantorovich potentials given by \eqref{eq:IntKantorovich}.
	Then $(1-a)^{p-1}\phi^a$ is a Kantorovich potential from $\mu^a$ to $\mu^1$ for all $a\in[0,1)$.
	
	Let $\bar\mu^a\coloneqq \mu^{1-a}$ be the reverse curve.
	Then $\bar\phi^0\coloneqq -\phi^1$ is a Kantorovich potential from $\bar\mu^0$ to $\bar\mu^1$.
	Denote by $\bar\phi^a$ the family of intermediate-time potentials associated with $(\bar\mu^a)_{a\in[0,1]}$.
	Then $\bar\phi^{1-a}=-\phi^a$ on $\spt(\mu^a)$.
\end{lemma}
Finally, it is worth mentioning that for $p = 2$, when $(M,g)$ is a smooth Riemannian manifold and $\mu \ll \mathrm{vol}_g$, the optimal transport map from $\mu$ to $\nu$ is given by $x\mapsto \exp_x(-\nabla \phi(x))$, where $\phi$ is a Kantorovich potential. See \cite{Mccann2001} or \cite[Theorem 2.33]{AmbrosioGigli2013}.

\subsection{Heat kernels on \texorpdfstring{$\rcd(K,\infty)$}{RCD(K,∞)} spaces}\label{sec:heatkernel}
The \emph{Cheeger energy} is a convex and l.s.c. functional on $L^2(X,\m)$, that is defined for $f\in L^2(X,\m)$ by
\[
\mathrm{Ch}(f)\coloneqq \inf\left\{\liminf_{k\to\infty}\frac{1}{2}\int_X \mathrm{lip}(f_k)^2\d \m:f_k\in\mathrm{Lip}(X),f_k\to f\text{ in $L^2(X,\m)$}\right\},
\]
where $\mathrm{lip}(f)(x)\coloneqq \limsup_{y\to x}\frac{|f(y)-f(x)|}{d(x,y)}$ denotes the local Lipschitz constant of $f$. 
The Sobolev space $W^{1,2}(X,d,\m)$ is defined as the finiteness domain of the Cheeger energy.
For all $f\in W^{1,2}(X,d,\m)$, $\ch(f)$ can be represented via the \emph{minimal relaxed gradient} $|\nabla f|$ of $f$ by $\ch(f)=\frac12\int_X |\nabla f|^2\d \m$; see \cite{AGSInvention}.

The Laplacian operator $\Delta f$ is defined as the element of minimal norm in the subdifferential of $\ch$ at $f$, see \cite{AGSInvention}. 
The space $(X,d,\m)$ is called \emph{infinitesimally Hilbertian} if $\mathrm{Ch}$ is a quadratic form on $L^2(X,\m)$. 
In this case, $\ch$ is a Dirichlet form and the Laplacian is the corresponding generator characterised by $\ch(u,v)=-\int_X \Delta u\cdot v\d \m$ for all $u\in D(\Delta)$ and $v\in D(\ch)$.
The space is said to satisfy the \emph{Riemannian curvature-dimension condition} $\rcd(K,N)$ if it is infinitesimally Hilbertian and satisfies $\cd(K,N)$; cf. \cite{Ambrosio-Gigli-Mondino-Rajala,AGSRCD2014,Gigli15MAMS}.

When the underlying space is infinitesimally Hilbertian, the heat flow induces a family of linear maps $P_t\colon L^2(X,\m)\to L^2(X,\m)$ for $t\in(0,\infty)$, see for more details in e.g. \cite[Section 5]{Gigli-Pasqualetto-Book}.
If in addition the space satisfies $\rcd(K,\infty)$, for each $t>0$ the operator $P_t$ admits an integral kernel $p_t$, as a symmetric non-negative function on $X\times X$ s.t. 
\[
P_t f(x)=\int_X f(y)p_t(x,y)\d \m(y),\quad \text{$\m$-a.e. $x\in X$}
\]
for all $f\in L^2(X,\m)$. 

We call $p_t$ the heat kernel on $X$ and denote for each $x\in X$ and $t>0$ by $p_t(x,\d\m)$ the probability measure $p(x,y)\m(\d y)$.
On weighted Riemannian manifolds, the heat kernel is the minimal positive fundamental solution to the equation $(\partial_t-\Delta_y)p_t(x,y)=0$.

The following useful weighted integral estimate for heat kernels is due to \cite{Tamanini-19HK}.
\begin{theorem}[\cite{Tamanini-19HK}]\label{thm:weightedInt}	Let $(X,d,\m)$ be an $\rcd(K,\infty)$ mms for $K\in \R$.
	 For each $t>0$, $x\in X$ and $D>2$, there exists $p'\in(1,2)$ and $C_D$ both depending only on $D$ s.t. 
	\begin{equation}\label{ineq:weightedInt}
		\int p^2_t(x,y)\exp\left(\frac{d^2(x,y)}{Dt} \right)\d \m (y)\leq \frac{1}{\m(B(x,\sqrt{2t}))}\exp\left(C_D+\frac{2p'Kt}{(2-p')(e^{2Kt}-1)}\right),
	\end{equation}
	for all $t>0$.
	In particular, $p_t(x,\cdot)\in L^2(X,\m)$ for all $t>0$ and $x\in X$.
\end{theorem}

\iffalse
\begin{corollary}
		Let $(X,d,\m)$ be a $\rcd(K,\infty)$ mms with $\spt(\m)=X$.
	Then the function $(t,x)\mapsto \log \m(B(x,t))$ is locally bounded on $(0,\infty)\times X$.
\end{corollary}
\begin{proof}

	For any $x\in X$ and $D>4$, by the Cauchy-Schwarz inequality
\begin{align}
	1&=\left(\int p_t(x,z)\d \m(z)\right)^2\leq \int p^2_t(x,z)\exp \left(\frac{d^2(x_0,z)}{Dt}\right)\d \m(z)\int \exp \left(-\frac{d^2(x_0,z)}{Dt}\right)\d \m(z).
\end{align}
For any $t<\frac{1}{2D\kappa}$, by \cref{thm:weightedInt}
\begin{align}
	\int p^2_t(x,z)\exp\left(\frac{d^2(x_0,z)}{Dt}\right)\d \m(z)	&\leq \int p_t^2(x,z)\exp\left(\frac{2d^2(x,z)}{Dt}\right)\d \m(z)\cdot \exp\left(\frac{2d^2(x_0,x)}{Dt}\right)\\
	&\lesssim_{D}\frac{1}{\m(B(x,\sqrt{2t}))} \exp\left(\frac{2d^2(x_0,x)}{Dt}\right)
\end{align}

\end{proof}
\fi

Let $ f_0\in L^2(X,\m)$ be a non-negative function with $\int_X f_0\d \m =1$ and let $[0, \infty) \ni t \mapsto f_t$ be the heat flow $P_tf_0$. 
It is known that for each $t>0$, $f_t$ remains non-negative, belongs to $L^2(X, \m) \cap D(\ent)$, and satisfies $\int_X f_t \d\m = 1$.
Moreover, it follows from gradient flow theory that the heat flow satisfies the following regularizing properties:

1) the curve $t\mapsto f_t$ on $L^2(X,\m)$, is gradient flow of Cheeger's energy starting from $f_0$, see e.g. \cite{AGSInvention,Gigli-Pasqualetto-Book}. As a consequence, for each $t>0$, $f_t\in D(\Delta)$ and 
\begin{equation}
	\ch(f_t)\leq \frac{\|f\|^2_{L^2}}{2t},\quad \|\Delta f_t\|^2_{L^2}\leq  \frac{\|f\|^2_{L^2}}{t^2};
\end{equation}
2) the curve $t\mapsto f_t\m$ on the Wasserstein space, realizes the metric gradient flow, as well as the EVI$_{K}$ gradient flow, of the Boltzmann entropy. Consequently, for all $\mu\in D(\ent)$ 
\begin{equation}\label{eq:EVI}
	I_K(t)\ent(f_t)+\frac{I_{K}(t)^2}{2}F(f_t)\leq I_K(t)\ent(\mu)+\frac{1}{2}W_2^2(\mu,f_t\m)
\end{equation}
where $I_K(t)\coloneqq \int^t_0 e^{Kr}\d r$ and $F(f_t)$ denotes the \emph{Fisher information} of $f_t$ given by $\int \frac{|\nabla f_t|^2}{f_t}\d \m$, see \cite{AGSRCD2014,Ambrosio-Gigli-Mondino-Rajala}.

We are mostly interested in cases where these functions are given by heat kernels. 
Thanks to \cref{thm:weightedInt}, above properties can be applied to $f_0=p_t(x,\cdot)$ for all $t>0$ and $x\in X$. 
The weighted integral estimate \eqref{ineq:weightedInt} further allows us to present the following entropy-moment estimate for heat kernels. 
The proof is provided in the appendix.
\begin{proposition}\label{prop:ME}
		Let $(X,d,\m)$ be an $\rcd(K,\infty)$ mms for $K\in \R$. 
		Then for all $t > 0$ and $x \in X$, $\|p_t(x,\cdot)\|_{L^2(X,\m)}$, $\|(p_t\log p_t)(x,\cdot)\|_{L^1(X,\m)}$, the second moment of $p_t(x,\cdot)$ and the Fisher information $F(p_t(x,\cdot))$ are finite.
	
		If in addition $(X,d,\m)$ is locally compact, then all these functions are locally bounded in $(t,x)$ i.e. they are bounded on $I\times U$ for any compact interval $I\subset (0,\infty)$ and any bounded subset $U\subset X$.
\end{proposition}

\begin{lemma}\label{lemma:uniintegrable}
	Let $(X,d,\m)$ be a proper $\rcd(K,\infty)$ mms for $K\in \R$.
	Let $\kappa\in\R$ and $\bar x\in X$ s.t.
	\[
	\tilde \m\coloneqq e^{-\kappa d^2(\bar x,x)}\m(\d x )\in\p(X). 
	\]
	Then for any bounded set $U\subset X$ and $f\in \Lip(X)$, the family of functions
	\[
	\{F_x(y)\coloneqq p_t(x,y)f(y)e^{\kappa d^2(\bar x,y)} :x\in U\}
	\] 
	is uniformly integrable in $L^1(X,\tilde \m)$.
\end{lemma}
\begin{proof}
	By the de la Vallée-Poussin theorem, it suffices to show the family $\{\varphi(|F_x|)\}_x$ is bounded in $L^1(X,\tilde \m)$ for $\varphi(a)=a\log a$.
	As $\tilde \m$ has finite total mass and the negative part of $\varphi(|F_x|)$ is uniform bounded by $e^{-1}$, it is enough to show the integral of $\varphi(|F_x|)$ has a uniform upper bound.
	This follows from a direct computation:
	\begin{align}
		\int_X\varphi(|F_x|)(y)\d\tilde \m(y)&=\int \left( \log p_t(x,y)+\log |f(y)|+\kappa d^2(\bar x,y) \right)p_t(x,y)|f(y)|\d \m(y)\\
		&\lesssim_{\kappa,f} \int (p_t(x,y)+1+d^2(\bar x,y))p_t(x,y)(d(\bar x,y)+1)\d \m(y)\\
		&\lesssim_{\bar x, U}\int (p_t(x,y)+1+d^2(x,y))p_t(x,y)(d(x,y)+1)\d \m(y)\\
		&\lesssim_{K,t} \m(B(x,\sqrt{2t}))^{-1}
	\end{align}
where the last inequality comes from \cref{thm:weightedInt}.
Since $(X,d)$ is proper, $\m(B(x,\sqrt{2t}))$ has a uniform positive lower bound for all $x\in U$.
\end{proof}
The following dimension-free estimates are particularly useful when working on spaces without any finite dimensionality assumption.
They are collected from 
\cite{Bakry-Gentil-Ledoux14}(Theorem 5.5.2 and Remark 5.6.2, see also their generalisation to $\rcd(K,\infty)$ mms cf. \cite[Corollary 4.4]{Tamanini-19HK} and \cite[Corollary 4.7]{AGS15}).
\begin{theorem}[\cite{Bakry-Gentil-Ledoux14,AGS15,Tamanini-19HK}] 
	Let $(X,d,\m)$ be an $\rcd(K,\infty)$ mms.
	For any positive function $f\in L^2(X)$ and $t>0$, we have
	\begin{enumerate}
		\item logarithmic Sobolev inequality:
		\begin{equation}\label{ineq:LSI}
			P_t(f\log f)-P_t f\log P_t f\geq I_{2K}(t)\frac{|\nabla P_t f|^2}{P_t f},\quad I_K(\tau)\coloneqq \int^\tau_0 e^{Kr}\d r;
		\end{equation}
		\item log-Harnack inequality: for every $x\in X$ and $\m$-a.e. $y\in X$
		\begin{equation}\label{ineq:LHI}
			\int_M p_t(x,z)\log p_t(x,z)\d\m(z)\leq \log p_{2t}(x,y)+\frac{d^2(x,y)}{4I_{2K}}.
		\end{equation}
		\item Gaussian lower bound: as a consequence of \eqref{ineq:LHI}, if $\m(X)<\infty$ then
		\begin{equation}\label{ineq:GaussLB}
			p_{2t}(x,y)\geq \m(X)^{-1} \exp\left(-\frac{d^2(x,y)}{4I_{2K}}\right).
		\end{equation}
	\end{enumerate}
\end{theorem}

\subsection{Gradient estimates on \texorpdfstring{$\rcd(K,\infty)$}{RCD(K,∞)} for unbounded Lipschitz functions}\label{sec:gradientestimate}
We denote by \( \Lip(X) \) the space of Lipschitz functions on \( X \) and by \( \Lipb(X) \) the subspace of bounded Lipschitz functions.

On an $\rcd(K,\infty)$ mms $(X,d,\m)$, it is known, by Kuwada \cite{Kuwadaduality} for the local Lipschitz constant, and by Ambrosio--Gigli--Savar\'e \cite{AGSRCD2014} for the minimal weak relaxed gradient with extension to all exponents by Savar\'e \cite{Savare14}, that for any $f\in \Lipb(X)\cap D(\ch)$, $t>0$ and $p\in[1,\infty]$, the following gradient estimates hold pointwise on $X$:
\begin{align}
		\lip P_t f&\leq e^{-Kt}\big(P_t(|\lip f|^p)\big)^{1/p}, & p\in[1,\infty)\label{ineq:G'_p}\tag{$G'_p$}\\
	\Lip (P_t f)& \leq e^{-Kt}\Lip(f),& p=\infty;\label{ineq:G'_infty}\tag{$G'_\infty$}
\end{align}
and $\m$-a.e. on $X$
\begin{align}
		|\nabla P_t f| &\leq e^{-Kt}\big(P_t(|\nabla f|^p)\big)^{1/p}, & p\in[1,\infty)\label{ineq:G_p}\tag{$G_p$}\\
	|\nabla P_t f|& \leq e^{-Kt}\|\nabla f\|_{\infty},& p=\infty.\label{ineq:G_infty}\tag{$G_\infty$}
	\end{align}

When considering Kantorovich potentials for optimal transports, boundedness or integrability can not in general be expected.  
Therefore, it becomes necessary to extend the gradient estimates to possibly unbounded Lipschitz functions.
	
We note that the heat flow can be extended to all Lipschitz functions.  
Indeed, for any $t > 0$ and $f \in \Lip(X)$, $P_t f$ can be defined pointwisely at each $x \in X$ by the integral  
\[
P_t f(x) \coloneqq \int_X p_t(x,y) f(y) \, \mathrm{d}\m(y),
\]  
which is always finite due to the finite moments of the heat kernel.  
By \cite[Theorem 6.1]{AGSRCD2014}, $P_t f$ is Lipschitz and satisfies \eqref{ineq:G'_infty}.

Recall that one can define weak upper gradients for functions in the \emph{local Sobolev class} (see \cite[Chapter 2.3]{Gigli15MAMS} or \cite[Chapter 2.1]{Gigli-Pasqualetto-Book}).  
More precisely, we define $S^2_{\mathrm{loc}}(X,d,\m)$ to be the set of all Borel functions $f \colon X \to \mathbb{R}$ s.t. for every bounded set $B \subset X$, there exists a function $f_B \in D(\mathrm{Ch})$ satisfying $f = f_B$ $\m$-a.e. on $B$.  

For any $f \in S^2_{\mathrm{loc}}(X,d,\m)$, we define the function $|\nabla f|$ on $X$ by setting, for each bounded set $B \subset X$,
\[
|\nabla f| \coloneqq |\nabla f_B| \quad \text{where } f_B = f \text{ $\m$-a.e. on } B,
\]
which is well-defined by the locality of the minimal relaxed gradient. Moreover, $|\nabla f|$ is a weak upper gradient of $f$ (see \cite[Proposition 2.1.32]{Gigli-Pasqualetto-Book}).

In particular, for any $f \in \Lip(X)$, the function $|\nabla f|$ is well-posed, belongs to $L^\infty(X,\m)$ and satisfies $|\nabla f| \leq \lip(f)$ $\m$-a.e. on $X$.

\begin{proposition}\label{prop:Gradientestimate}
	Let $(X,d,\m)$ be an $\rcd(K,\infty)$ mms. 
	For any $f\in \Lip(X)$, gradient estimates \eqref{ineq:G'_p} and \eqref{ineq:G_p} hold for all $p\in[1,\infty]$ and $t>0$.
	Moreover, it holds pointwise on $X$ that
	\begin{equation}\label{ineq:lipLip}
\lip P_tf\leq e^{-Kt}P_t(|\nabla f|).
	\end{equation}
\end{proposition}
\begin{proof}
We divide the proof into three steps.
	
	\textbf{Step 1}: gradient estimate $(G'_1)$ for Lipschitz functions.
	
	The proof follows in line of the proof of \cite[Proposition 3.1]{Kuwadaduality}.
	Define for any $r>0$ and $z\in X$ the function 
	\[
	G_r(z)\coloneqq \sup_{ w\in B(z,r)\setminus\{z\}}\frac{|f(z)-f(w)|}{d(z,w)}.
	\]
	For any $x,y\in  X$, by \cite[Theorem 4.4]{Savare14}, there exists a coupling $\sigma_{x,y}$ between $p_t(x,\d\m)$ and $p_t(y,\d\m)$ s.t. 
	\[
	\|\sigma_{x,y}\|_{L^\infty(X\times X)}\leq e^{-Kt }d(x,y).
	\]
	Then for $r=d(x,y)$ we have 
	\begin{align}
		|P_tf(x)-P_tf(y)|&\leq \int_{X\times X}|f(z)-f(w)|\d\sigma_{x,y}(z,w)\\
		&\leq 	\|\sigma_{x,y}\|_{L^\infty(X\times X)}\int_{X\times X}G_r(z)\d\sigma_{x,y}(z,w)\\
		&=	\|\sigma_{x,y}\|_{L^\infty}P_t(G_r)(x)\leq e^{-Kt }d(x,y)P_t(G_r)(x).
	\end{align}
As $G_r$ is uniformly bounded by $\Lip(f)$, applying the dominated convergence theorem with $y\to x$ obtains $(G'_1)$.
The estimate \eqref{ineq:G'_p} for $p\in(1,\infty)$ follows by H\"older's inequality.
\smallskip

\textbf{Step 2}: \eqref{ineq:lipLip} for bounded Lipschitz functions.

For any fixed $\bar x\in X$ and each $m\in\N$, we define the cut-off function
\begin{equation}\label{eq:cut-off}
		\chi_m(x)\coloneqq \big(1-\mathrm{dist}(x,B(\bar x,m))\big)\vee 0,
	\end{equation}
which is a Lipschitz function with bounded support.
Then for any $f\in\Lip_b(X)$, $f_m\coloneqq f\cdot \chi_m\in D(\ch)\cap \Lip_b(X)$ for all $m\in\N$.
Thus applying the classical gradient estimate \cite[Theorem 4.1]{Savare14} to $f_m$ yields that 
\begin{align}\label{ineq:15/07/25-1}
	\lip P_tf_m\leq e^{-Kt}P_t(|\nabla f_m|)
	\end{align}
holds pointwisely on $X$.
At each $x\in X$, using \eqref{ineq:15/07/25-1} and the estimate $(G'_1)$ for the Lipschitz function $f(1-\chi_m)$ with the locality of weak upper gradient gives
\begin{align}
	\lip P_tf&\leq \lip P_t f_m+\lip P_t (f (1-\chi_m))\\
	&\leq e^{-Kt}P_t(|\nabla f_m|)+e^{-Kt}P_t (\lip f(1-\chi_m))\\
	&\leq e^{-Kt}P_t (|\nabla f|\chi_m)+e^{-Kt}P_t\left[2f|\nabla \chi_m|+(1-\chi_m)\lip f \right].
\end{align}
Since $f,\lip f\in L^\infty(X,\m)$, applying the dominated convergence theorem with $m\to \infty$ concludes the inequality \eqref{ineq:lipLip}.
\smallskip

\textbf{Step 3}:  \eqref{ineq:lipLip} for Lipschitz functions.

For any $f\in\Lip(X)$ and each $n\in\N$, we define the truncation functions of $f$ by 
\[
f_n\coloneqq f\wedge n\vee (-n), \quad f^+_n\coloneqq f\vee n,\quad f^-_n\coloneqq f\wedge (-n).
\]
For all $x,y\in X$, we have
\begin{align}
P_tf(x)-P_tf(y)&=\int_{\{f>n\}}+\int_{\{|f|\leq n\}}+\int_{\{f<-n\}} (p_t(x,z)-p_t(y,z))f(z)\d \m(z)\\
&=P_tf^+_n(x)-P_tf^+_n(y)+P_tf_n(x)-P_tf_n(y)+P_tf^-_n(x)-P_tf^-_n(y),
\end{align}
which in particular implies
\begin{equation}\label{ineq:splitlip}
	\lip P_tf\leq \lip P_tf_n+\lip P_tf^+_n+\lip P_t f^-_n.
\end{equation}
Applying \eqref{ineq:lipLip} to the bounded Lipschitz function $f_n$ and $(G'_1)$ to Lipschitz functions $f^\pm_n$ implies that
\begin{align}
		\lip P_tf\leq e^{-Kt }P_t |\nabla f|+e^{-Kt}P_t \left( \lip f_n^-+\lip f^+_n  \right)
\end{align}
holds everywhere on $X$ for all $n\in\N$.
Since $\lip f^\pm_n$ are uniformly bounded and monotone decreasing to $0$ as $n\to \infty$, taking $n$ to $\infty$ concludes \eqref{ineq:lipLip} for $f$, and \eqref{ineq:G_p} follows together with the almost everywhere comparison $|\nabla P_t f|\leq \lip P_t f$.
\end{proof}

\begin{lemma}\label{lemma:Sobolev-Lip}
	Let $(X,d,\m)$ be an $\rcd$ mms.
	Let $f\colon X\to \R$ be a measurable function admitting a weak upper gradient $g\in L^\infty(X,\m)$.
	Then $f$ has a Lipschitz representative $\tilde f$ with $\Lip(\tilde f)\leq \|g\|_{L^\infty}$.
\end{lemma}
\begin{proof}
	For each $m\in\N$ and $\bar x\in X$, consider the cut-off function $\chi_m$ on $X$ given by \eqref{eq:cut-off} and for all $n\in\N$ we define the truncation functions of $f$ by 
	\[
f_n\coloneqq f\wedge n\vee (-n),\quad 	f_{n,m}\coloneqq \chi_m\cdot f_n
	\]
	Then by the locality and the Leibniz rule of weak upper gradient, for all $n,m$ the function $f_{n,m}\in D(\ch)$, and admits a weak upper gradient given by
	\[
	|\nabla\chi_m ||f_n|+\chi_m g\leq n+\|g\|_{L^\infty}\eqqcolon L_n.
	\]
	Applying \cite[Theorem 6.2]{AGSRCD2014} shows for all $n,m\in\N$ that $f_{n,m}$ is $L_n$-Lipschitz up to changing to a continuous representative. 
	The same has to hold also for $f_n$ since $\chi_m\nearrow 1$ as $m\to\infty$.
	
	Now we claim that for all $n$, $\Lip(f_n)\leq \|g\|_{L^\infty}$, and thus passing $n$ to $\infty$ yields that $f$ is Lipschitz with $\mathrm{Lip}(f)\leq \|g\|_{L^\infty}$.
	
	For any $x_0,x_1\in X$ and $r>0$, consider $\m^i_{r}\coloneqq \frac{\m\llcorner B(x_i,r)}{\m(B(x_i,r))}$ for $i=0,1$ and $\pi_r$ the optimal dynamical plan between $\m^0_r $ and $\m^1_r$.
	Then take $m$ sufficiently large to contain all curves in $\spt(\pi)$ and use the weak upper gradient estimate 
	\begin{align}
	\left|	\aint{B(x_0,r)}f_n-	\aint{B(x_1,r)}f_n\right|&=	\left|	\aint{B(x_0,r)}f_n\chi_m-	\aint{B(x_1,r)}f_n\chi_m\right|\\
	&\leq \int \int^1_0 g(\gamma^a)|\dot\gamma^a|\d a\d\pi_r(\gamma)\\
	&\leq \|g\|_{L^\infty}(d(x_0,x_1)+2r).
	\end{align}
As it is known that $f_n$ is Lipschitz, the integral averages converge when $r$ goes to $0$. 
Then the claim follows.
\end{proof}

\subsection{Splitting spaces via vanishing Hessian}
For a smooth manifold, the existence of a non-constant smooth function whose Hessian vanishes everywhere implies that the space splits.
\begin{lemma}\label{lemma:Euclidean}
	Let $(M,g)$ be a complete manifold, and let $\phi$ be a non-constant function on $M$. The following are equivalent:
	\begin{enumerate}
		\item[(1)] $M$ is a direct product of $\R$ and a complete Riemannian manifold $N$, and there exist $a,b\in\R$ s.t. $\phi(t,x)=at+b$ for all $(t,x)\in \R\times N$;
		\item[(2)] $\phi$ is smooth and $\nabla^2 \phi\equiv 0$ on $M$;
		\item[(3)] $\phi$ is linear along each geodesic.
	\end{enumerate}
\end{lemma}
\begin{proof}
	The implication from (1) to both (2) and (3) is trivial.
	The direction from (2) to (1) follows from \cite{Catino-Mantegazza-Lorenzo}.
	It is worth mentioning that we do not need any curvature assumption here.
	Since $\nabla^2\phi$ is everywhere vanishing and $\phi$ is non-constant, $\phi$ has no critical point and $(M,g)$ is non-compact.
	Otherwise, $\phi$ is harmonic and has to be constant by Liouville's theorem.
	By the same argument of Theorem 1.1 of \cite{Catino-Mantegazza-Lorenzo}, given $\Sigma$ any regular level set of $\phi$, $\phi$ depends only on the signed distance to $\Sigma$ and $M$ is a warped product $\R\times_{\phi'}\Sigma$ i.e.
	\[
	g=\d r\otimes \d r+(\phi')^2g_{\Sigma}.
	\]
	Recall that $\nabla^2\phi=\phi''\d r\otimes \d r+\phi'\nabla^2 r$ and $\nabla^2 r(\partial_r,\partial_r)=0$. 
	Hence by assumption 
	\[
	0\equiv\nabla^2\phi(\partial_r,\partial_r)=\phi'' .
	\]
	In other words, $\phi$ is linear and $M$ is a direct product of $\R$ and $\Sigma$.
	
	Now we show the implication (3)$\Rightarrow$(2).
	If $\phi$ is linear along each geodesic, it is in particular geodesically convex. 
	Then by Alexandrov’s theorem, see e.g. \cite[Theorem 14.1]{Villani}, for a.e. $x\in M$, $\phi$ is differentiable and there exists a symmetric operator $H_x\colon T_xM\to T_xM$ s.t. 
	\begin{equation}\label{eq:alexandrov}
		\phi(\exp_x(v))=\phi(x)+\langle\nabla \phi(x),v\rangle+\frac{\langle H_xv,v\rangle}{2}+o(\|v\|^2)
	\end{equation}
	as $v\to 0$. Now for any $0\neq v\in T_x M$ since $\phi$ is linear along $t\mapsto \exp_x(tv)$, there is $c_v\in\R$ s.t. $\phi(\exp_x(tv))=\phi(x)+c_vt$ for all $t $ small.
	Plugging this into \eqref{eq:alexandrov} gives
	\[
	c_v t=t\langle \nabla \phi(x),v\rangle+\frac{t^2}{2}\langle H_x v,v\rangle +o(t^2)
	\]
	which forces $c_v=\langle \nabla \phi(x),v\rangle$ and $\langle H_x v,v\rangle =0$. 
	As $v$ is arbitrary, $H_x=0$ and its trace $\Delta \phi(x)=0$ for a.e. $x\in M$. 
	
	Again by Alexandrov’s theorem, the absolutely continuous part of the distributional Laplacian of $\phi$ is $0$, whose singular part is non-negative.
	Notice that $-\phi$ is also linear along each geodesic. Applying the previous argument to $-\phi$ yields that $\phi$ is a distributional solution to the Laplace equation. 
	By Weyl's lemma, $\phi$ is smooth. With this regularity we conclude $\nabla^2 \phi\equiv 0$ on $M$.
\end{proof}

\begin{remark}\label{remark:linear}
	Recall that a convex function on $\R$ admits left and right derivatives, which are non-decreasing. 
	Therefore, in \cref{lemma:Euclidean} if we know $\phi$ is geodesically convex and continuous, the linearity in (3) can be relaxed as follows: for every geodesic $(\gamma^a)_{a\in[0,1]}$, 
	\begin{align}
		\liminf_{h\searrow 0}\frac1h (\phi(\gamma^h)-\phi(\gamma^0))\geq 	\limsup_{h\searrow 0}\frac1h(\phi(\gamma^1)-\phi(\gamma^{1-h}).
	\end{align} 
\end{remark}

When considering weighted manifolds, if additional curvature assumptions are imposed, we can establish the splitting for the metric-measure structure.
\begin{proposition}\label{prop:splitmesaure}
	Let $(M,g,\m)$ be a weighted complete manifold with $\m\coloneqq e^{-f}\mathrm{vol}_g$ and $\Ric_{\m}\geq K$.
	Let $\phi$ be a non-constant smooth function on $M$ that satisfies
	\[
	\Ric_{\m}(\nabla\phi)\equiv K|\nabla\phi|^2,\quad \nabla^2\phi\equiv0
	\]
	on $M$. 
	Then there exist a weighted complete Riemannian manifold $(N,g_N,\n)$ and functions $f_1\colon \R\to \R$ and $f_2\colon N\to \R$ s.t. 
	\begin{equation}
		(M,g,\m)=(\R,|\cdot|,e^{-f_1}\mathcal{L}^1)\times (N,g_N,\n),\quad \n\coloneqq e^{-f_2}\mathrm{vol}_{g_N}.
	\end{equation}
	Moreover, $f_1''\equiv K$ on $\R$ and $\Ric_{\n}\coloneqq \Ric_{g_N}+\nabla^2 f_2\geq K g_N$.
\end{proposition}
\begin{proof}
	By \cref{lemma:Euclidean}, the metric space $(M,g)$ splits as a direct product $\R\times N$ for some complete manifold $(N,g_N)$, and $\phi$ depends only on the first coordinate i.e. there exists a function $\phi_1\colon\R\to \R$ s.t. for any $(r,w)\in M=\R\times N$, $\phi(r,w)=\phi_1(r)$.
	Denote $\partial_r=\nabla\phi$. Then $\Ric(\partial_r )=0$ by the splitting of the metric.
	By assumption
	\begin{equation}\label{eq:partial_r}
		\nabla^2 f(\partial_r,\partial_r)=\Ric_{\m}(\partial_r,\partial_r)-\Ric(\partial_r,\partial_r)=K g(\partial_r,\partial_r).
	\end{equation}
	For any $x=(r,w)\in M$ and $v\in T_w N$ denote $\tilde v=(0,v)\in T_xM$. 
	Then the condition
	\[
	\Ric_\m(\partial_r+\lambda \tilde v)\geq Kg(\partial_r+\lambda \tilde v,\partial_r+\lambda \tilde v),\quad \forall \lambda\in\R
	\]
	with \eqref{eq:partial_r} and the splitting of $g$, implies 
	\begin{align}
		\lambda^2\left[(\Ric_{g_N}-Kg_N)(v,v)+\nabla^2 f(\tilde v,\tilde v) \right]+2\lambda \nabla^2f (\partial_r ,\tilde v)\geq 0
	\end{align}
	for all $\lambda\in\R$. This forces 
	\[
	\Ric_{g_N}(v,v)+ \nabla^2 f(\tilde v,\tilde v)\geq Kg_N(v,v), \quad \nabla^2 f(\partial_r, \tilde v)=0
	\]
	for all $v$ and $x\in M$.
	The above demands $f$ to split, $\Ric_n\geq K$ and $f''\equiv K$ restricting on $\R$.
\end{proof}

%%%%%%%%%%%%%%%

%%%%%%%%%%%
\section{Sharp gradient estimates and sharp Wasserstein contractions}
In this section, we aim to establish the equivalence between the cases in which equality is attained in the Wasserstein contractions and in the Bakry--Émery gradient estimates. 
To this end, we introduce the notion of sharpness conditions for both types of estimates, for general exponents $p \in [1, \infty]$.

\begin{definition}\label{def:sharp}
	Let $(X,d,\m)$ be an $\rcd(K,\infty)$ mms. For $p\in[1,\infty]$, we say that $(X,d,\m)$ attains
	\begin{enumerate}
		\item \emph{sharp $p$-Wasserstein contraction}, denoted by $(SC^K_p)$ for short, if there exist $x\neq y\in X$ and $t>0$ s.t. 
		\begin{equation}\label{eq:SCp}
			W_p(p_t(x,\d\m),p_t(y,\d\m))=e^{-Kt}d(x,y);
		\end{equation}
		\item \emph{sharp \(p\)-gradient estimate}, denoted by \((SG^K_{p})\), if there exist \(t > 0\) and a non-constant function \(f\in \Lip(X)\) s.t. the following holds $\m$-a.e.
		\begin{align}
			|\nabla P_t f| &= e^{-Kt}\big(P_t(|\nabla f|^p)\big)^{1/p}, && p\in[1,\infty)\label{ineq:SG^K_p}\\
			|\nabla P_t f|& = e^{-Kt}\|\nabla f\|_{L^\infty},&& p=\infty.\label{eq:SG_infty}
		\end{align}
	\end{enumerate}
\end{definition}

\begin{lemma}\label{lemma:sharpconditions}
		Let $(X,d,\m)$ be an $\rcd(K,\infty)$ mms for $K\in\R$. Then for all $p\in(1,\infty)$
		\begin{enumerate}
			\item\label{item1} $(SC^K_1)\Rightarrow (SC^K_p)\Rightarrow (SC^K_\infty)$;
		\item\label{item2} $(SG^K_\infty)\Leftrightarrow (SG^K_p)\Rightarrow (SG^K_1)$;
		\item $(SG^K_p)\Leftrightarrow$ $\exists f\in \Lip(X)$ which is non-constant and attains equality in \eqref{ineq:G'_p} at some $x\in X$.
		\end{enumerate}
\end{lemma}
\begin{proof}
	\cref{item1} follows directly from H\"older's inequality.
	For \cref{item2}, we assume that  \((SG^K_p)\) is attained by $f\in \Lipb(X)$ for some \(p \in (1, \infty)\) and \(t>0\).
	Then it holds by $(G_1)$ and H\"older's inequality that
	\begin{equation}\label{ineq:25/07/25-1}
e^{Kt }|\nabla P_tf|=	\big(P_t(|\nabla f|^p)\big)^{1/p}\geq P_t(|\nabla f|)\geq e^{Kt }|\nabla P_tf|
	\end{equation}
In particular, the Hölder inequality attains equality.
Thus \( |\nabla f| \) is constant almost everywhere, and by the everywhere equality in \eqref{ineq:25/07/25-1}, so is \( |\nabla P_t f| \). This shows \((SG^K_\infty)\) and $(SG^K_1)$.

	The last statement can be proved in a similar way.
	Here we show only the ``$\Leftarrow$" direction.
	
	By $(G'_1)$ and H\"older's inequality, the assumption implies that
	\[
(P_t (\lip f)^p)(x)=|e^{Kt}\lip P_t f|^p (x)\leq (P_t (\lip f))^p(x)\leq (P_t (\lip f)^p)(x).
	\]
	Therefore by the equality case of H\"older's inequality, $\lip f$ is constant a.e.
	Further by \eqref{ineq:lipLip}
\[
\lip P_tf(x)\leq e^{-Kt}P_t(|\nabla f|)(x)\leq e^{-Kt }P_t(\lip f)(x)
\]
	where again equalities are attained everywhere, forcing $|\nabla f|\equiv \lip f$ $\m$-a.e. on $X$.
	The proof is complete by the Sobolev-to-Lipschitz property \cref{lemma:Sobolev-Lip}.
\end{proof}

%%%%%%%%%%%%%%%%%%

\subsection{The duality theorem}\label{sec:dualitythm}
The goal of this subsection is to prove the following theorem.
\begin{theorem}\label{thm:dualitysharp}
	Let $(X,d,\m)$ be a proper $\rcd(K,\infty)$ mms for $K\in\R$. Then
	\begin{enumerate}
			\item\label{item4} $(SG^K_\infty)\Rightarrow (SC^K_1)$, provided that the heat flow is ultracontractive. The same implication holds without the ultracontractivity assumption if $X$ is a smooth manifold.
		\item\label{item3} $(SC^K_2)\Rightarrow (SG^K_\infty)$, and the same implication holds without the properness assumption if $\m(X)<\infty$.
	\end{enumerate}
\end{theorem}
\begin{proof}[Proof of the statement \eqref{item4}]
		Let $f$ be a non-constant Lipschitz function achieving \eqref{eq:SG_infty} for some $t>0$. Set $g\coloneqq P_t f$.
		Without loss of generality, we may assume
		\begin{equation}\label{eq:assumption}
		|\nabla g|\equiv e^{-Kt}\Lip(f)=e^{-Kt}\quad \text{$\m$-a.e. on $X$.}
		\end{equation}
	\smallskip
	
	\textbf{Part 1}: smooth case.
	
	We assume first that $X$ is a smooth manifold.
	For any $x\in X$, consider the gradient flow curve of $g$:
	\[
	\dot\gamma^a=-\nabla g(\gamma^a),\quad \gamma^0=x.
	\]
	Then by Kantorovich duality
	\begin{align}
		W_1(p_t(x,\d\m),p_t(\gamma^1,\d \m))&\geq P_tf(x)-P_tf(\gamma^1)= g(\gamma^0)-g(\gamma^1)\\
		&=\int^1_0 |\nabla g|^2(\gamma^\tau)\d \tau\overset{\eqref{eq:assumption}}{= } e^{-Kt}\int^1_0 |\nabla g|(\gamma^\tau)\d\tau\\
		&\geq e^{-Kt} \mathrm{Length}(\gamma|_{[0,1]} )\geq e^{-Kt}d(x,\gamma^1),\label{ineq:25/7/7-1}
	\end{align}
	which implies $(SC^K_1)$.
	\smallskip
	
	\textbf{Part 2}: construction of  gradient flow curves on $\rcd(K,\infty)$ mms.
	
Without loss of generality, assume $f(\bar{x}) = 0$ for some $\bar{x} \in X$.  
For each $n \in \mathbb{N}$, consider the truncation $f_n \coloneqq f \chi_n$, where
\begin{equation}
	\chi_n(x) \coloneqq \left(1 - \frac{\mathrm{dist}(x, B(\bar{x}, n))}{n} \right) \vee 0.
\end{equation}
Since $\Lip(\chi_n) \leq \frac{1}{n}$ and $|f| \leq 2n$ on $\spt(\chi_n)$, it follows that $\Lip(f_n) \leq 3$.  
Moreover, writing $\Delta P_t f_n = P_{t/2} \Delta P_{t/2} f_n$, and using the ultracontractivity and regularising properties of the heat flow, we see that $P_t f_n$ is a test function (see e.g. \cite[Definition 6.1.7]{Gigli-Pasqualetto-Book}).
	
Then by \cite{Ambrosio-Trevisan14} (see Theorem 9.6 together with Lemma 7.4 and Theorem 7.6), for any probability density $\rho \in L^\infty(X, \m)$ and each $n$, there exists a measure $\pi_n \in \mathcal{P}(C([0,1]; X))$\footnote{The path measure $\pi_n$ is, in fact, the lift of a curve $(\mu^a)_{a \in [0,1]}$ of probability measures solving the continuity equation with initial data $\mu^0 = \rho \m$ and velocity field given by the gradient of $P_t f_n$.} such that $(e_0)_\# \pi_n = \rho \m$, and:

\begin{enumerate}[(i)]
	\item\label{itemi} $\pi_n$ is concentrated on absolutely continuous curves $\gamma$ with $|\dot\gamma^a| = |\nabla P_t f_n|(\gamma^a)$ for almost every $a$;
	
	\item\label{itemii} for $\pi_n$-almost every $\gamma$, it holds for all $0 \leq a < b \leq 1$ that
	\begin{equation}
		P_t f_n(\gamma^a) - P_t f_n(\gamma^b) = \int_a^b |\nabla P_t f_n|^2(\gamma^\tau) \, \mathrm{d}\tau.
	\end{equation}
\end{enumerate}
In particular, there exist a point $x \in X$ and a sequence of absolutely continuous curves $(\gamma_n)_{n \in \mathbb{N}}$ with $\gamma_n(0) = x$ s.t. for each $n \in \mathbb{N}$, the curve $\gamma_n\in\spt(\pi_n)$ and satisfies \ref{itemi} and \ref{itemii}.

By the sharpness assumption and the inequality \eqref{ineq:lipLip}, we obtain the following two-sided bounds on $|\nabla P_tf_n|$, valid $\m$-a.e. on $X$:
\begin{align}
	G^-_n&\coloneqq e^{-Kt}\left( P_t(|\nabla f|)-P_t(|\nabla (f-f_n)|) \right)\label{ineq:20/07/25-2}\\
&\leq |\nabla P_t f|-\lip P_t (f-f_n)\leq	|\nabla P_t f_n|\leq |\nabla P_t f|+\lip P_t (f-f_n)\\
	&\leq e^{-Kt}\left( P_t(|\nabla f|)+P_t(|\nabla (f-f_n)|) \right)\eqqcolon G^+_n\label{ineq:20/07/25-3}
\end{align}
The functions $G_n^\pm$ are continuous on $X$ and uniformly bounded for all $n$.

Therefore, without loss of generality (even if the two-sided bounds for $|\nabla P_t f_n|$ fail on a non-negligible set, one can show that they hold along $\pi_n$-a.e. curve by the same argument as in \cite[Proposition 5.10]{AGSInvention}), we may assume that for each $n$, $\gamma_n$ satisfies
\begin{equation}\label{ineq:16/07/25-1}
		P_t f_n(\gamma^0_n)-P_tf_n(\gamma^1_n)\geq \int^1_0G^-_n(\gamma^\tau_n)^2\d \tau,\quad |\dot\gamma^a_n|\leq G^+_n(\gamma^a_n),\quad \forall a\in(0,1).
\end{equation}

Note that the curves $(\gamma_n)_n$ are Lipschitz with a uniform Lipschitz constant. 
Hence, by the properness of $X$, there exists a continuous curve $\gamma$ with $\gamma^0 = x$ s.t. up to a subsequence, $\gamma_n \to \gamma$ uniformly.
	\smallskip
	
	\textbf{Part 3}: sharp Wasserstein contraction along $\gamma$.
		
		As $P_tf$ is continuous and $P_tf_n\to P_tf$ pointwisely, 
		\begin{equation}\label{ineq:20/07/25-1}
			P_tf(\gamma^0)-P_tf(\gamma^1)=\lim_{n\to \infty} P_tf_n(\gamma^0)-P_tf_n(\gamma^1_n)+P_tf_n(\gamma^1_n)-P_tf(\gamma^1_n).
		\end{equation}
	By \cref{lemma:uniintegrable}, there exists a sequence $ \varepsilon_n\searrow 0$ s.t. for all $\tau\in[0,1]$
	\begin{equation}\label{ineq:20/07/25-4}
		P_t|\nabla (f-f_n)|(\gamma^\tau_n)\leq \int_{B(\bar x,n)^c}p_t(\gamma^\tau_n,y)(f(y)+|\nabla f|)\d \m(y)\leq \varepsilon_n.
	\end{equation}
Similarly, one can show that $|P_tf_n(\gamma^1_n)-P_tf(\gamma^1_n)|\leq \varepsilon_n$.
By the assumption of $(SG^K_\infty)$, $P_t|\nabla f|$ equals $1$ everywhere on $X$.
	Therefore, combining above inequalities gives
	\begin{align}
		P_tf(\gamma^0)-P_tf(\gamma^1)&\overset{\eqref{ineq:20/07/25-1}}{=}\lim_{n\to \infty} P_tf_n(\gamma^0)-P_tf_n(\gamma^1_n)\\
		&\overset{\eqref{ineq:16/07/25-1}}{\geq} \lim_{n\to\infty}\int^1_0G^-_n(\gamma^\tau_n)^2\d \tau\geq \lim_{n\to\infty}\left(\int^1_0 G^-_n(\gamma^\tau_n)\d\tau\right)^2\\
		&\overset{\eqref{ineq:20/07/25-4}}{=} \lim_{n\to\infty}\left(\int^1_0 G^+_n(\gamma^\tau_n)\d\tau\right)^2\\
		&\overunderset{\eqref{ineq:16/07/25-1}}{\eqref{ineq:20/07/25-3}}{\geq }\lim_{n\to\infty}\int^1_0 e^{-Kt}P_t|\nabla f|(\gamma^\tau_n)\d\tau\cdot \int^1_0 |\dot\gamma^\tau_n|\d\tau\\
		&\geq  e^{-Kt}d(\gamma^0,\gamma^1)\int^1_0P_t|\nabla f|(\gamma^\tau)\d\tau \overset{(SG^K_\infty)}{=}e^{-Kt}d(\gamma^0,\gamma^1).
	\end{align}
This verifies that \eqref{ineq:25/7/7-1} holds at the point \( x = \gamma^0 \) and the endpoint \( \gamma^1 \).
	\end{proof}

	\begin{proof}[Proof of the statement \eqref{item3}]
	Assume that $(SC^K_2)$ is attained at some pair of points $x_0 \neq y_0 \in X$ and at some time $t > 0$. 
	Let $(\gamma^a)_{a\in[0,1]}$ be a constant speed geodesic from $x_0$ to $y_0$.
	By \cref{prop:geodesic}, $[0,1]\ni a\mapsto\mu^a\coloneqq p_t(\gamma^a,\d\m)$ is a Wasserstein geodesic.
	Let $\phi$ be a Kantorovich potential relative to an optimal transport plan $\sigma$ between $\mu^0$ and $\mu^1$. 
	
	We divide the proof into two steps. 
	
 \iffalse  
  As in \cref{prop:geodesic}, the curve $a\mapsto p_t(\gamma^a,\d\m)$ is a $W_p$-geodesic.
   Owing to the ultracontractivity of the heat flow, the heat kernel $p_t$ is Lipschitz on $X\times X$ (see \cite[Proposition 6.4]{AGSRCD2014}). 
   Moreover, by the log-Harnack inequality \eqref{ineq:LHI} and the local boundedness of the entropy \cref{prop:ME}, the heat kernel is locally bounded away from zero.
  More precisely, for any bounded set $X_0\subset X$, there exist $c_0,C_0>0$ s.t.
   \[
 c_0\leq p_t(\gamma^a,y)\leq C_0,\quad \forall y\in X_0, a\in[0,1].
   \]
   Thus, Theorem 10.3 in \cite{AGSInvention} applies.
    In particular, for any optimal plan $\sigma$ between $p_t(x,\d\m)$ and $p_t(y,\d\m)$, it holds that
   \begin{equation}\label{eq:metricBrenier}
   	d(\tilde x,\tilde y)=|\nabla f|(\tilde x),\quad \text{$\sigma$-a.e. $(\tilde x,\tilde y)$}.
   \end{equation}
\fi
\smallskip

\textbf{Step 1}: $\phi$ is Lipschitz on $X$.
%Furthermore, following the notation in \cite{AGSInvention}, we define
%\begin{align}
%	D^+(x,a)\coloneqq \sup_{(y_n)}\limsup_{n\to\infty} d(x,y_n),\quad 
%	D^-(x,a)\coloneqq \inf_{(y_n)}\liminf_{n\to\infty} d(x,y_n)
%	\end{align}
%where in both cases, $(y_n)$ varies among all minimizing sequences of the infimum in \eqref{eq:HLsemigroup2}.

We claim that for any bounded subset $U\subset X$, there exists $c_U>0$ s.t. $p_t(x_0,y)>c_U$ for all $y\in U$.
Indeed, when $(X,d)$ is proper, the spatial function $x\mapsto \ent(p_t(x,\d\m))$ is locally bounded by \cref{prop:ME}, and hence by the log-Harnack inequality \eqref{ineq:LHI} the heat kernel has positive lower bound over any bounded set.
The same holds by applying the Gaussian lower estimate \eqref{ineq:GaussLB} if $\m$ is finite.

Then the claim on the heat kernel allows us to apply \cite[Lemma 10.1]{AGSInvention}, which shows with the slope estimate \cite[Proposition 3,9]{AGSInvention} that $\phi$ admits a weak upper gradient s.t. $|\nabla \phi|(x)\leq d(x,y)$ for $\sigma$-a.e. $(x,y)$.
 \iffalse By \cite[Propostion 3.1]{AGSInvention}, for every $x\in X$, the quantities $D^\pm(x,a)$ coincide for a.e. $a\in [0,1]$. 
By Fubini's theorem, this implies that for a.e. $a\in[0,1]$, $D^-(x,a)=D^+(x,a)$ for $\m$-a.e. $x\in X$.
Therefore, since the sharp Wasserstein contraction is fulfilled along the entire curve $(\mu^a)$ by \cref{prop:geodesic}, we may, up possibly replacing the transport with the transport between $\mu^a$ and $\mu^1$ for some $a > 0$, assume that $D^-(x,0) = D^+(x,0)$ for $\m$-a.e. $x \in X$.
Then the slope estimate \cite[Propositon 3.4]{AGSInvention} yields for $\m$-a.e. $x\in X$
\begin{equation}\label{ineq:slope}
	\lip \phi (x)=\lip Q_1(-\phi^c)(x)\leq D^+(x,1)=D^-(x,1)
\end{equation}
\fi
	
	By the $\infty$-Wasserstein contraction \cite[Theorem 4.4]{Savare14}, there exists a coupling $\tilde{\sigma}$ between $\mu^0$ and $\mu^1$ s.t. $\|\tilde \sigma\|_{L^\infty(X\times X)}\leq e^{-Kt }d(x_0,y_0)$.
	Then the following chain of inequalities must in fact be equalities:
	\[
	e^{-Kpt}d(x_0,y_0)^p\geq \int d^p(x, y)\d\tilde \sigma\geq \int d^p( x, y)\d \sigma=W^p_p(\mu^0,\mu^1)=e^{-Kpt}d(x_0,y_0)^p.
	\]
	In other words, $\tilde \sigma $ is also an optimal transport plan w.r.t. $W_2$ distance.
    By the uniqueness of optimal plans (see \cite{Rajala-Sturm14}), it follows that $\sigma = \tilde{\sigma}$.
    As a result, we have $|\nabla \phi|(x)\leq e^{-Kt}d(x_0,y_0)$ for $\mu^0$-a.e. $x$ and hence also for $\m$-a.e. $x\in X$ since $\mu^0$ and $\m$ are mutually absolutely continuous.
	
	Finally, by the Sobolev-to-Lipschitz property \cref{lemma:Sobolev-Lip}, we conclude $\phi\in \Lip(X)$, up to changing to a continuous representative.
	\smallskip
	
	\textbf{Step 2}: $f\coloneqq -\phi$ attains the sharp $\infty$-gradient estimate.	
	To make the argument slightly more general, we first work with a general exponent $p \in (1,\infty)$ and emphasize when we specialize to $p = 2$.
	
	Set $L(s)\coloneqq p^{-1}s^p$ and its Legendre conjugate $L^*$ as $L^*(s)\coloneqq q^{-1}s^q$ with $q$ the H\"older dual of $p$.
	 We define the Hopf-Lax semigroup for $f$ as in \eqref{eq:HLsemigroup1} by
	\[
	Q_a f(x)\coloneqq \inf_{y\in X}\left[  f(y)+aL\left(\frac{d(x,y)}{a}\right) \right].
	\]
As $f$ is Lipschitz, $(a,x)\mapsto Q_af(x)$ is Lipschitz in $[0,\infty)\times X$ (see e.g. \cite{Balogh12HJ}).
By \cref{prop:Gradientestimate}, for each $t>0$ we have pointwise on $X$ that
\begin{equation}\label{ineq:lip-gradient}
	\lip (P_tQ_a f)\leq e^{-Kt}P_t |\nabla Q_a f|.
	\end{equation}

By the sharp contraction assumption and Kantorovich duality, we have:
\begin{align}\label{eq:SC^K_pproof}
	\frac{1}{p}e^{-Kt p}d^p(x_0,y_0)=\frac{1}{p}W^p_p(\mu^0,\mu^1)&=-P_t f(x_0)+P_t (Q_1(f))(y_0).
\end{align}
Repeating the proof of \cite[Proposition 3.7]{Kuwadaduality} obtains
\begin{align}
	&-P_t f(x_0)+P_t (Q_1f)(y_0)=\int^1_0 \partial_a \left( P_t (Q_af)(\gamma^a)\right)\d a\\
	\overset{(I_1)}{\leq}& \int^1_0 d(x_0,y_0)\lip (P_t(Q_af))(\gamma^a) -P_t (L^*(|\lip Q_a f|))(\gamma^a)\d a\\
	= &\int^1_0 d(x_0,y_0)\lip (P_t(Q_af))(\gamma^a) -P_t \left( \frac{|\lip Q_a f|^q}{q} \right)(\gamma^a)\d a\\
	\overset{(I_2)}{\leq} & \int^1_0 d(x_0,y_0)\lip (P_t(Q_af))(\gamma^a) -\frac{1}{q} \left(e^{Ktq} |\lip P_t Q_a f|^q \right)(\gamma^a)\d a\\
	 \overset{(I_3)}{\leq} &\int^1_0 L(W_p(\mu^0,\mu^1)))\d a =\frac{1}{p}W^p_p(\mu^0,\mu^1),
\end{align}
where in $(I_2)$ we applied the pointwise $q$-gradient estimate $(G'_q)$ to $Q_af$ for each $a$, and in $(I_3)$ we used the Legendre transform $L(s)=\sup_{t
\geq 0}\{ts-L^*(t)\}$.
We briefly comment on $(I_1)$. There, we applied $\lip(P_t(Q_a f))$ as an upper gradient of $P_t(Q_a f)$ along the geodesic $\gamma$. 
Moreover, for $p = 2$, the map $(a,x) \mapsto Q_a f(x)$ solves a Hamilton–Jacobi equation for a.e. $a \in (0,1)$ and $\m$-a.e. $x \in X$; see \cite[Theorem 3.5]{AGSInvention}.

With \eqref{eq:SC^K_pproof}, equality in the gradient estimate we used in $(I_2)$ must be attained for a.e. $a\in[0,1]$. 
Therefore, the proof is complete by \cref{lemma:sharpconditions}.
\iffalse %%%%%%%%%%%%%%%
	Assume that there are $x\neq y\in M$ and $t>0$ satisfying $\eqref{eq:SCp}$ for $p=1$. 
	By Kantorovich duality, there is a $1$-Lipschitz function $f$ s.t.
	\begin{align}\label{eq:30/06-01}
		e^{-Kt}d(x,y)=W_1(p_t(x,\d\m),p_t(y,\d\m))=P_t f(x)-P_tf(y).
	\end{align}
	For any $s\in[0,t)$, set $g\coloneqq P_s f$ which is $e^{-Ks}$-Lipschitz and  satisfies
	\[
	|\nabla P_{t-s}g|\leq e^{-K(t-s)}P_{t-s}|\nabla g| .
	\]
	As $P_t$ maps contractively $L^\infty(X)$ into $C_b(X)$, $P_{t-s}|\nabla g| $ is continuous on $X$.
	Then by the formula (4.17) in \cite{Savare14}, we obtain:
	\begin{equation}
		\lip(P_{t-s}g)\leq e^{-K(t-s)}P_{t-s}|\nabla g|\leq e^{-Kt}.
	\end{equation}
	Now apply $\lip(P_{t-s}g)$ as an upper gradient of $P_{t-s}g$ along the unit-speed geodesic $(\gamma^a)_{a\in[0,r]}$ from $x$ to $y$. 
	This yields:
	\begin{align}
		P_t f(x)-P_tf(y)&= P_{t-s}g(x)-P_{t-s}g(y)\leq \int^r_0 \lip(P_{t-s}g)(\gamma^a)\d a\\
		&\leq \int^r_0 e^{-K(t-s)}P_{t-s}|\nabla g|(\gamma^a)\d a\leq e^{-Kt}d(x,y)  .
	\end{align}
Combining this with \eqref{eq:30/06-01}, we conclude that equality must hold everywhere. In particular, $P_{t-s}|\nabla g|=e^{-Ks}$ along $\gamma$, which forces $|\nabla g|\equiv e^{-Ks}$ on $X$ due to the bound $|\nabla g|\leq e^{-Ks}$.
This in turn yields that $|\nabla f|\equiv 1$ and $(SG^K_\infty)$ is attained.
%%%%%%%%%%%%%%
\fi
\end{proof}

The author expects that the restriction to $p = 2$ in the statement \eqref{item3} is merely technical, and that the implication $(SC^K_p) \Rightarrow (SG^K_\infty)$ should hold for all $p \in [1,\infty)$ under the same assumption.  
This implication is known to be valid at least when $(X,d,\m)$ has finite dimension.

	We say a mms is a \emph{finite dimensional} $\rcd(K,\infty)$ space if it satisfies both $\rcd(K,\infty)$ and $\rcd(K',N)$ for some $K'\in \R$ and $N<\infty$.
	\begin{corollary}\label{cor:sharpN}
		Let $(X,d,\m)$ be a finite dimensional $\rcd(K,\infty)$ mms.
		Then $(X,d,\m)$ satisfies $(SG^K_\infty)$ if and only if it satisfies $(SC^K_p)$ for some (and hence all) $p\in[1,\infty)$.
	\end{corollary}
\begin{proof}
Notice that the restriction of $p = 2$ is required only in the following three places in the proof of \cref{thm:dualitysharp} \cref{item3}:
	\begin{enumerate}[(i)]
		\item the uniqueness of optimal transport plans;
		\item the existence of a bounded weak upper gradient for the Kantorovich potential;
		\item the use of the Hamilton--Jacobi equation in \( (I_2) \) of Step 2.
	\end{enumerate}
	The uniqueness of optimal transport plans for general costs of the form $d^p(x,y)$ has been established in \cite{Cavalletti-Huesmann15}, and the Hopf-Lax semigroup as a sub-solution to the Hamilton--Jacob equation for general Hamiltonian is shown in \cite[Theorem 2.1]{Balogh12HJ} (where all assumptions are satisfied on finite dimensional $\rcd$ spaces).

It remains to verify that, for general $ p \in (1,\infty)$, the Kantorovich potential $\phi$ w.r.t. $W_p$ admits a weak upper gradient that is bounded.
This can be proved by following the same route as in \cite{AGSInvention}, using the properties of the Hopf--Lax semigroup from \cite[Section 3]{Cavalletti-McCann21}.

We only provide a sketch here.  
First, one can replicate the argument of \cite[Proposition 3.4]{AGSInvention} (cf. also \cite[Proposition 3.4]{AGS13RMI}) to obtain the slope estimate for Kantorovich potential $\phi $:
	\begin{equation}\label{ineq:pslope}
		\lip^+ \phi (x)=\lip^+ Q_1(-\phi^c)(x)\leq D^-(x,1)^{p-1},
	\end{equation}
where $\lip^+\phi(x)$ and $D^-(x,1)$ are defined as 
\[
 \lip^+\phi(x)\coloneqq \limsup_{y\to x} \frac{[\phi(y)-\phi(x)]\vee 0}{d(x,y)}, \quad D^-(x,1)\coloneqq \inf_{n}\liminf_{n\to \infty} d(x,y_n)
 \]
with the infimum taken over all minimizing sequences $(y_n)$ in the definition of Hopf--Lax semigroup \eqref{eq:HLsemigroup1}.
The conclusion on the weak upper gradient then follows from the same arguments in \cite[Lemma 10.1]{AGSInvention}.
\end{proof}
\subsection{Remarks on weighed Euclidean spaces}\label{sec:rmkrigidity}
It is natural to ask whether the sharp Wasserstein contraction can actually be attained.  
For each $n \in \N$ and $k \in \R$, we denote by $\R^n_k$ the $n$-dimensional weighted Euclidean space
\begin{equation} \label{eq:weightedEuc}
( \R^n, |\cdot|, e^{-\frac{k}{2}|x|^2} \mathcal{L}^n(\d x) ).
\end{equation}
For every $k \in \R$, $\R^n_k$ satisfies $\rcd(k,\infty)$ and heat kernels are bounded when $k\leq 0$; see \eqref{eq:Mehler} below.

It is not obvious even in the case $n = 1$ whether the Wasserstein distance between heat kernels attains equality in the contraction estimate when $k \neq 0$.  
However, by \cref{thm:dualitysharp}, it suffices to check the sharp gradient estimates with a suitable Lipschitz function.

It has been shown in \cite{Han-JMPA21} for $\R^1_k$, $k=1$ and for a linear function $f$ that $|\nabla P_t f|=e^{-kt}|\nabla f|$.
The same also holds for negative $k$. 
Indeed, the heat kernel on $\R^1_k$ for $k=-2$ is given by 
\begin{equation}\label{eq:Mehler}
	p_t(x,y) = \frac{1}{(2\pi \sinh 2t)^{1/2}} \exp\left( \frac{2xye^{-2t} - x^2 - y^2}{1 - e^{-4t}} - t \right)
\end{equation}
see \cite[Example 9.18]{Grigoryan09HK}.
A direct computation for $f(x)\equiv x$ shows
\begin{align}
	P_tf(y) &=(2\pi \sinh 2t)^{-1/2}e^{-t}\int_\R x\exp\left( \frac{2xye^{-2t} - x^2 - y^2}{1 - e^{-4t}} \right) e^{x^2}\d x\\
	&= e^{-2t}(\pi (1-e^{-4t}))^{-1/2}\int x\exp\left( -\frac{(e^{-2t}x-y)^2}{1-e^{-4t}}\right)\d x\\
	=&\frac{1}{\sqrt{\pi}}e^{2t}y\int e^{-x^2}\d x=e^{2t}y, \quad \forall y\in\R.
\end{align}
This confirms $|\nabla P_t f|=e^{-kt}|\nabla f|$ for $k=-2$ and the case of general $k\in\R$ follows by a scaling argument.
\begin{proposition}\label{prop:Eucleadian}
	Let $(\R^n,|\cdot|,e^{-f}\mathcal{L}^n)$ be a weighted Euclidean space with $\nabla^2f \equiv k\mathrm{Id}$ for some $k\in\R$.
	Then for all $x,y\in \R^n$, $t>0$ and $p\in[1,\infty]$, 
	\begin{equation}
		W_p(p_t(x,\cdot),p_t(y,\cdot) )=e^{-kt}|x-y|.
	\end{equation} 
\end{proposition}
\begin{proof}
	By writing the space as a product of one-dimensional weighted Euclidean spaces and with the tensorization property of heat kernels, we only need to show for $(\R,|\cdot|,e^{-f}\mathcal{L}^1)$ with $f(x)=
	\frac{k}{2}x^2+ax$.
	For $k\neq 0$, the space is isomorphic to $\R^1_k$ by translation, where we have shown the sharp $\infty$-Bakry--\'Emery estimate is attained.
	
	For $k=0$ and $a\neq 0$, denote by $p_t$ and $\tilde p_t$ the heat kernel on the unweighted and weighted Euclidean space, respectively.
	By \cite[Theorem 9.15]{Grigoryan09HK}, one has the following relation
	\begin{equation}\label{eq:26/03-1}
		\tilde p_t(x,z)=p_t(x,z)e^{-a^2t}e^{-a(x+z)},
	\end{equation}
	which allows us to check the sharp $\infty$-Bakry--\'Emery estimate for linear function as before. 
	The statement then follows from \cref{thm:dualitysharp}.
\end{proof}

\subsection{Rigidity results}\label{sec:apptoduality}
	For any $K\in\R$, we say that an $\rcd(K,\infty)$ mms $(X,d,\m)$ \emph{splits off a line} if there exists a mms $(X',d',\m')$, either satisfying $\rcd(K,\infty)$ or being a singleton, s.t. $(X,d,\m)$ is isomorphic as a mms to $(X',d',\m')\times \R^1_K$.% where $\R^1_K$ is given by \eqref{eq:weightedEuc} and 
%\begin{itemize}
%	\item  $(X',d',\m')$ satisfies $\rcd(K,\infty)$ if $N\geq 2$;
%	\item  $(X',d',\m')$ is a point if $N\in[1,2)$.
%\end{itemize}

The following result is a direct consequence of \cref{cor:sharpN} and \cite{ABS19Gafa} (see Theorem 2.1 and Remark 2.3). 
\begin{corollary}\label{cor:sharpK=0}
	Let $(X,d,\m)$ be an $\rcd(0,N)$ mms for $N<\infty$.
	Then the following are equivalent:
	\begin{enumerate}
		\item\label{item:split} $(X,d,\m)$ splits off a line;
		\item $(X,d,\m)$ satisfies $(SC^0_p)$ for some $p\in[1,\infty)$;
		\item $(X,d,\m)$ satisfies $(SG^0_p)$ for some $p\in[1,\infty]$.
	\end{enumerate}
\end{corollary}
For spaces having finite total mass, applying \cref{thm:dualitysharp} to the rigidity results established in \cite{Han-JMPA21} (see Proposition 3.13 and Corollary 3.8) yields the following.
\begin{corollary}
	Let $(X,d,\m)$ be an $\rcd(K,\infty)$ mms for $K\in\R$. 
	Assume that $\m(X)<\infty$. 
	If $(X,d,\m)$ satisfies $(SC^K_2)$, then $K>0$ and the space splits off a line.
\end{corollary}
In other words, on spaces of finite total mass, the $2$-Wasserstein contraction can not be sharp for $K \leq 0$.
The next theorem provides an additional rigidity regarding the curvature bound for finite dimensional spaces:
\begin{theorem}\label{thm:Kmustbe0}
	Let $(X,d,\m)$ be either 
	\begin{itemize}
		\item an $\rcd(K,N)$ mms with $K\in\R$ and $N<\infty$; or
		\item a finite dimensional $\rcd(K,\infty)$ mms with $K\geq 0$.
	\end{itemize}
	If it satisfies either $(SC^K_p)$ or $(SG^K_p)$ for some $p\in[1,\infty)$, then $K=0$ and the space splits off a line.
\end{theorem}
We postpone the proof to the end of the subsection, as it relies on two auxiliary results on spaces satisfying sharp Wasserstein contraction.
\begin{lemma}\label{prop:geodesic}
	Let $(X,d,\m)$ be an $\rcd(K,\infty)$ mms for some $K\in\R$.
	If $(X,d,\m)$ satisfies $(SC^K_2)$ for some $x\neq y\in X$ and $t>0$, then for all $s\in (0,t)$, any geodesic $(\gamma^a)_{a\in[0,1]}$ (if exists) from $x$ to $y$, and all $0\leq a<b\leq 1$
	\begin{equation}\label{eq:sharpWgeo}
		W_2(p_s(\gamma^a,\d\m ),p_s(\gamma^b,\d\m))=e^{-Ks}d(\gamma^a,\gamma^b).
	\end{equation}
	In particular, for all $s\in [0,t]$, $(p_s(\gamma^a,\d \m))_{a\in[0,1]}$ is a $2$-Wasserstein geodesic.
\end{lemma}
\begin{proof}
	We use the short-hand notation $\mu^a_s\coloneqq p_s(\gamma^a,\d\m)$ for $s\in[0,t]$ and $a\in[0,1]$.
	By the Wasserstein contraction, we have for all $0\leq a\leq b\leq 1$
	\begin{align}
		W_2(\mu^0_t,\mu^1_t)\leq& W_2(\mu^0_t,\mu^a_t)+W_2(\mu^a_t,\mu^b_t) +W_2(\mu^b_t,\mu^1_t)\\
		\leq &e^{-K(t-s)}\left(W_2(\mu^0_s,\mu^a_s)+W_2(\mu^a_s,\mu^b_s) +W_2(\mu^b_s,\mu^1_s)\right)\\	
		\leq& e^{-Kt}\left(d(\gamma^0,\gamma^a)+d(\gamma^a,\gamma^b)+d(\gamma^b,\gamma^1)\right)=e^{-Kt}d(x,y).
	\end{align}
	By assumption all inequalities in the above have to be equalities.
\end{proof}

\begin{proposition}\label{lemma:linearEnt}
	Let $(X,d,\m)$ be either 
	\begin{itemize}
		\item a weighted smooth manifold satisfying $\rcd(K,\infty)$; or
		\item a finite dimensional $\rcd(K,\infty)$ mms,
	\end{itemize}
	for some $K\in\R$.
	Let $(\gamma^a)_{a\in[0,1]}$ be a geodesic along which \eqref{eq:sharpWgeo} holds. 
	Then $[0,1]\ni a\mapsto \ent(p_s(\gamma^a,\d\m))$ is a quadratic function i.e.
	\begin{equation}\label{eq:entquadratic}
		\ent(p_s(\gamma^a,\d\m))-	KW^2_2(p_s(x,\d\m),p_s(y,\d\m)) \cdot \frac{a^2}{2}
	\end{equation}
	is linear in $a$.
\end{proposition}
\begin{proof}
	Denote for all $a\in[0,1]$ $\mu^a\coloneqq p_s(\gamma^a,\d\m)$ and $\rho^a(z)\coloneqq p_s(\gamma^a,z)$ for short.
	Let $(\phi^a)_{a\in[0,1]}$ be the family of intermediate-time Kantorovich potentials given as in \cref{lemma:interPotential}.
	
	As the sharp contraction estimate \eqref{eq:sharpWgeo} is fulfilled along the whole curve $\gamma$, we know from the proof of \cref{thm:dualitysharp}\eqref{item3} that $\phi^a$ is Lipschitz for all $a\in[0,1]$ with $\Lip(\phi^a)\leq e^{-Ks}d(\gamma^0,\gamma^1)$.
	\smallskip
	
	\textbf{Part 1}: $X$ is a smooth manifold.
	
	Rewriting the heat equation, we see for each $a\in[0,1]$ that $(p_\tau(\gamma^a,\d \m))_{\tau>0}$ solves the continuity equation with velocity field given by $-\nabla\log \rho^a_\tau$.
	Moreover, the quantity $\|\nabla\log\rho^a\|^2_{L^2(\mu^a)}$ coincides the Fisher information $F(\rho^a)$, which is locally bounded in time by \cref{prop:ME}.
	Hence, we may apply \cite[Theorem 23.9]{Villani} to obtain\footnote{Note that our sign convention for the Kantorovich duality differs from that used in \cite{Villani}.}
	\begin{align}\label{ineq:WdisalongHF}
		-\frac{\d}{\d s} \left(\frac12 W^2_2(p_s(x,\d\m),p_s(y,\d\m))\right)=\int_M \nabla\phi^0\cdot \nabla  p_s(x,\cdot)\d \m-\int_M \nabla\phi^1\cdot\nabla p_s(y,\cdot)\d \m.
	\end{align}
	Along the Wasserstein geodesic $(\mu^a)$, applying \cite[Theorem 23.14]{Villani} gives
	\begin{equation}\label{ineq:subDent1}
		\liminf_{h\searrow 0} \frac1h\left(\ent(\rho^{a+h})-\ent(\rho^a)\right)\geq \int_M -\nabla\phi^a\cdot \nabla \rho^a\d\m.
	\end{equation}
	Similarly, by considering the reverse curve $\bar\mu^a=\mu^{1-a}$ with the associated backward potentials $\bar\phi^a$ given by \cref{lemma:interPotential}, one obtains
	\begin{equation}\label{ineq:subDent2}
		\liminf_{h\searrow 0} \frac1h\left(\ent(\rho^{a-h})-\ent(\rho^{a})\right)\geq \int_M -\nabla\bar\phi^{1-a}\cdot \nabla \rho^a\d\m=\int_M \nabla\phi^a\cdot \nabla \rho^a\d\m.
	\end{equation}
	Thus combining \eqref{ineq:subDent1} and \eqref{ineq:subDent2}, with \cref{prop:geodesic} and $K$-convexity of entropy along Wasserstein geodesics, we have
	\begin{align}
		K\cdot W^2_2(p_s(x,\d\m),p_s(y,\d\m))&=	-\frac{\d}{\d s} \left(\frac12 W^2_2(p_s(x,\d\m),p_s(y,\d\m))\right)\\
		&=\int_M \nabla\phi^0\cdot \nabla  p_s(x,\cdot)\d \m-\int_M \nabla\phi^1\cdot\nabla p_s(y,\cdot)\d \m\\
		&\geq \limsup_{h\searrow 0} \frac1h\left(\ent(\rho^{1})-\ent(\rho^{1-h})\right)-\liminf_{h\searrow 0} \frac1h\left(\ent(\rho^{h})-\ent(\rho^{0})\right)\\
		&\geq 	K\cdot W^2_2(p_s(x,\d\m),p_s(y,\d\m)),
	\end{align}
	The function \eqref{eq:entquadratic} of $a$ then has to be linear by \cref{remark:linear}, and the inequalities in \eqref{ineq:subDent1} and \eqref{ineq:subDent2} have to be equalities.
\smallskip

\textbf{Part 2}: $(X,d,\m)$ is a finite dimensional $\rcd(K,\infty)$ mms.

The inequalities \eqref{ineq:WdisalongHF} and \eqref{ineq:subDent1} remain valid for a.e.\ $s$ in the $\operatorname{RCD}$ setting, as established in \cite[Theorems 6.3 and 6.5]{Ambrosio-Gigli-Mondino-Rajala}. 
Therefore, the proof follows by the same line of argument used in the smooth setting.

We remark that in \cite{Ambrosio-Gigli-Mondino-Rajala}, the results are stated under the assumption that one marginal of the transport plan has compact support. 
However, in our setting, all Kantorovich potentials $\phi^a$ are known to be Lipschitz, and the associated densities are given by heat kernels, which are uniformly bounded and satisfy Gaussian estimates. 
These properties allow the arguments in \cite{Ambrosio-Gigli-Mondino-Rajala} to extend naturally to our case.
\end{proof}
\begin{proof}[Proof of \cref{thm:Kmustbe0}]
By \cref{cor:sharpN}, it is enough to consider that $(X,d,\m)$ satisfies the sharp contraction condition $(SC^K_2)$.
As an $\rcd(K,N)$ space, $(X,d)$ is geodesic.
Let $(\gamma^a)_{a\in[0,1]}$ be a geodesic s.t. \eqref{eq:sharpWgeo} holds for some $s>0$.

The condition of $\rcd(K,N)$ strengthens $K$-displacement convexity by requiring the entropy to be $(K,N)$-convex along Wasserstein geodesics; see \cite{EKS15}.
In particular, the function $a\mapsto E(a)\coloneqq \ent(p_s(\gamma^a,\d\m))$ satisfies 
\begin{equation}\label{ineq:KNconvex}
	E''(a)\geq KW^2_2(p_s(\gamma^0,\d\m),p_s(\gamma^1,\d\m))+\frac{1}{N}E'(a)^2.
\end{equation}
By \cref{lemma:linearEnt}, $E'(a)\equiv 0$, which is possible only when $K=0$.
\smallskip

Next we show that the sharp gradient estimate $(SG^K_1)$  can not be attained on any finite dimensional $\rcd(K,\infty)$ mms $(X,d,\m)$ with $K>0$.
If equality were attained, by \cite[Proposition 3.13]{Han-JMPA21}, there exists an $\rcd(K,\infty)$ mms $(Y,d_Y,\m_Y)$ s.t. $(X,d,\m)$ is isomorphic to $\R^1_K\times (Y,d_Y,\m_Y)$.

Denote for any $a\in\R$ by $(\eta^a_t)_{t\geq 0}$ the heat flow on $\R^1_K$ starting from $\delta_a $.
As in Part 1 of the proof of \cref{thm:dualitysharp}\eqref{item4}, and using the fact that the linear function $a \mapsto a$ attains the sharp gradient estimate, it follows that the sharp Wasserstein contraction holds along the entire family $(\eta^a_t)_{a \in \mathbb{R}}$, i.e.
\[
W_1(\eta^a_t, \eta^b_t) = W_2(\eta^a_t, \eta^b_t)=e^{-K t} |a - b|.
\]

Now fix any $y \in Y$, and consider the curve of heat flows on $X=\R\times Y$
\[
\R\ni a\mapsto \mu^a\coloneqq p_t((a,y),\d\m).
\]
Observe that $a\mapsto \mu^a$ is a $2$-Wasserstein geodesic on $(X,d,\m)$.
Indeed, let $T^b_a:\R\to \R$ be the optimal transport map from $\eta^a$ to $\eta^b$. 
Then, the product map $(T^b_a, \mathrm{id}_Y)$ pushes $\mu^a$ forward to $\mu^b$, due to the tensorization of heat flows (see \cite[Theorem 6.19]{AGSRCD2014}).  
Moreover, this map is optimal by checking directly the cyclical monotonicity.

As a consequence,
\[
W_2(\mu^a,\mu^b)=W_2(\eta^a_t,\eta^b_t)=e^{-K t}|a-b|
\]
so the sharp Wasserstein contraction holds along the entire curve $(\mu^a)_{a \in \mathbb{R}}$.

Applying \cref{prop:geodesic} to every interval in $\mathbb{R}$, we conclude on $\R$ that 
\begin{equation}
	E(a)\coloneqq \ent(\mu^a)=\frac{K}{2}W^2_2(\mu^0,\mu^1)a^2+\text{a linear function of $a$}.
	\end{equation}
Now assume that $(X,d,\m)$ also satisfies $\rcd(K',N)$ for $K'\in\R$ and $N<\infty$. 
Then the $(K',N)$-convexity of the entropy implies (see \eqref{ineq:KNconvex})
\[
(K-K')W^2_2(\mu^0,\mu^1)\geq \frac{1}{N}(E'(a))^2,\quad \forall a\in\R.
\]
However, this leads to a contradiction since $|E'(a)| \to \infty$ as $a \to \infty$.
\end{proof}
We conjecture that the rigidity of curvature bounds in \cref{thm:Kmustbe0} holds for all finite dimensional $\rcd(K,\infty)$ and $K\in\R$.  
Indeed, at least in the smooth setting, the rigidity cases for sharp Wasserstein contraction w.r.t. both positive and negative $K$ exhibit the same structural phenomenon: the space splits off a factor $\R^1_K$, which does not carry any finite synthetic dimension; see \cref{thm:sharpness}.

\section{rigidity of Wasserstein contraction on smooth manifolds}\label{sec:proofsharpness}
In this section, we provide a direct proof of the following rigidity theorem on Wasserstein contraction in the smooth setting.
\begin{theorem}\label{thm:sharpness}
	Let $(M,g,\m)$ be a weighted smooth manifold satisfying \cref{asm:graHK} and $\Ric_\m\geq K$ for $K\in\R$.
	If equality in \eqref{eq:WassersteinK} holds for some $x\neq y\in M$ and $t>0$, then there exist a weighted complete manifold $(N,g_N,\n)$ and a function $f_1\colon \R\to \R$ s.t. 
	\begin{equation}\label{eq:splitting}
		(M,g,\m)=(\R,|\cdot|,e^{-f_1}\mathcal{L}^1)\times (N,g_N,\n).
	\end{equation}
	Moreover, $f_1''\equiv K$ on $\R$ and $\Ric_\n\geq K$.
\end{theorem}

\begin{assumption}\label{asm:graHK}
	A mms $(X,d,\m)$ verifies \cref{asm:graHK} if at least one of the following holds
	\begin{enumerate}
		\item the heat flow is ultracontractive i.e. $P_t$ is a bounded linear map from $L^1(X,\m)$ to $L^{2}(X,\m)$ for each $t>0$;
		\item $\m$ is finite i.e. $\m(X)<\infty$.
		%	\item for each $r>0$, there exists $x_0\in X$ and $c_0>0$ s.t. $\m(B(x,r))>re^{-d^2(x_0,x)}$.
	\end{enumerate}
\end{assumption}
We remark that this assumption is imposed for technical reasons (see \cref{lemma:nonsymintegral} below). 
It is always satisfied, for example, if \( K > 0 \) or if the underlying space satisfies an extra synthetic finite dimensionality condition.

Throughout, we consider only the weighted Laplacian and divergence, which for any smooth function $\phi$ on $(M,g,e^{-f}\mathrm{vol}_g)$ are given by $\Delta \phi=\mathrm{div}_\m \nabla \phi\coloneqq \Delta_g \phi -\nabla \phi\cdot \nabla f$ where $\Delta_g$ is the Laplace-Beltrami operator on $(M,g)$.
\subsection{Proof of \texorpdfstring{\cref{thm:sharpness}}{Theorem 4.1}}
The proof is divided into two parts: a formal discussion presenting the heuristics, followed by a rigorous argument.

\textbf{Part 1}: preparation and heuristics.

Assume that equality in \eqref{eq:WassersteinK} holds for some $x\neq y\in M$ and $t>0$.
Let $(\gamma^a)_{a\in[0,1]}$ be a geodesic from $x$ to $y$. 
Denote $\mu^a\coloneqq p_t(\gamma^a,\d\m)$ and $\rho^a(z)\coloneqq p_t(\gamma^a,z)$ for $a\in[0,1]$.

Let $(\phi^a)_{a\in[0,1]}$ be the family of intermediate-time potentials associated with $(\mu^a)_{a\in[0,1]}$, given by \eqref{eq:IntKantorovich}.
Note that \(\partial_a \rho^a\) is smooth on \(X\), and for each \(a \in (0,1)\) the function \(\phi^a\) solves the following elliptic PDE weakly on the smooth manifold \(X\):
\begin{equation}\label{eq:CE0323}
	\mathrm{div}_{\m}(\rho^a\nabla\phi^a)=\partial_a\rho^a,
\end{equation}
which by elliptic regularity ensures that \(\phi^a\) is smooth on \(X\).  
Moreover, \(\phi^a\) is smooth in \(a\) as the solution of the PDE depends smoothly on its coefficients. 
Theorem 3.6 in \cite{AGSInvention}, together with the smoothness, implies that the map $(a,x)\mapsto \phi^a(x)$ is a classic solution to following Hamilton-Jacobi equation
\begin{equation}\label{eq:HJ}
	\frac{\d}{\d a} \phi^a(x)-\frac12|\nabla\phi^a|^2(x)=0.
\end{equation}
By \cref{lemma:linearEnt} (with also its proof), we have 
\begin{equation}\label{eq:linearityofent}
	\frac{\d }{\d a}\ent(p_t(\gamma^a,\d\m))=-\int_M \nabla\phi^a\cdot \nabla \rho^a\d\m= KW^2_2(\mu^0,\mu^1)\cdot a+\text{constant}.
\end{equation}
Formally applying Otto's calculus (see e.g. \cite[Section 3]{Otto-Villani00} for a formal argument and \cite[Section 4]{Erbar-Li-Schultz25} for a rigorous treatment on compact manifolds) along the geodesic $(\mu^a)$ yields
\begin{align}\label{eq:01/04-1}
	&KW^2_2(\mu^0,\mu^1)=-\frac{\d }{\d a}\int_M \nabla\phi^a\cdot \nabla \rho^a\d\m\\
\overunderset{\eqref{eq:CE0323}}{\eqref{eq:HJ}}{=}&-\int_M \nabla \frac{|\nabla \phi^a|^2}{2}\nabla\rho^a+\nabla\phi^a\cdot \nabla (\nabla\cdot(\rho^a\nabla\phi^a))\d \m\\
	=&\int_M \Delta \frac{|\nabla\phi^a|^2}{2}\rho^a-\nabla\Delta \phi^a\cdot \nabla \phi^a \rho^a\m\\
	=& \int_M \Ric_\m(\nabla\phi^a,\nabla\phi^a)+\|\nabla^2 \phi^a\|_{\mathrm{HS}}^2 \d\mu^a\\
\geq &\int_M K|\nabla \phi^a|^2+\|\nabla^2 \phi^a\|_{\mathrm{HS}}^2 \d\mu^a=	KW^2_2(\mu^0,\mu^1)+\int_M\|\nabla^2 \phi^a\|_{\mathrm{HS}}^2 \d\mu^a.
\end{align}
Then the assumption of \cref{prop:splitmesaure} is satisfied and the space splits.

However, justifying the integrations by parts used above is difficult due to the lack of integrability of higher-order derivatives of $\phi^a$. 
To pass the derivative of the entropy along the Wasserstein geodesic into the potentials $\nabla \phi^a$, we need to localize the transport to bounded subsets.
The key idea is not to apply an exhaustion argument directly on $M$, but rather on the path space associated with the optimal dynamical plan $\pi$ for $(\mu^a)_{a \in [0,1]}$.
\medskip

\textbf{Part 2}: Rigorous implementation via an exhaustion argument on the path space.

Notice first that $\nabla \phi^a\cdot \nabla \rho^a\in L^1(M,\m)$ since
\[
	\int |\nabla\rho^a\cdot \nabla\phi^a|\d\m\leq \left(\int \frac{|\nabla\rho^a|^2}{\rho^a}\d \m\right)^{1/2}\left(\int |\nabla\phi^a|^2\rho^a\d \m\right)^{1/2}<\infty.
\]
Then by \cref{lemma:nonsymintegral},
\[
\|\partial_a \rho^a\|_{L^1}=\int_M |\nabla_x p_t(\gamma^a,y)\cdot \dot\gamma^a|\d \m(y)<\infty
\] 
and then with \eqref{eq:CE0323}, $\Delta\phi^a\rho^a\in L^1(X,\m)$.
Therefore the following integration by parts is valid\footnote{Indeed, take a sequence of functions $\chi_n\in C^{\infty}_c(M)$ with $0\leq \chi_n,|\nabla\chi_n|\leq 1$ and $\chi_n\nearrow 1$ on $M$. One shows the integration by parts first with $\chi_n$ and then passes $n\to \infty$ with the dominated convergence theorem.}
\begin{equation}\label{ineq:23/03-1}
	\int \Delta \phi^a(\eta^a) \d\pi(\eta)=\int\Delta \phi^a\rho^a \d \m=	-\int_M \nabla\phi^a\cdot \nabla \rho^a\d\m.
\end{equation}

Denote by $\pi$ the optimal dynamical plan of $(\mu^a)_{a\in[0,1]}$.
We have for all $a\in[0,1)$ that
\begin{equation}\label{eq:direction}
	\dot{\eta}^a=-\nabla\phi^a(\eta^a), \quad \text{ $\pi$-a.e. $\eta$}.
\end{equation} 
For any fixed $x_0\in M$ and each $L>0$, consider $\Gamma_L$ the set of geodesics $(\eta^a)_{a\in[0,1]}$ on $M$ with $\eta^0,\eta^1\in B(x_0,L)$, $\pi_{L}\coloneqq \frac{\pi\llcorner \Gamma_{L}}{\pi(\Gamma_L)}$ and $\mu^a_L\coloneqq (e_a)_{\#}\pi_{L}$.
Note for each $a\in[0,1)$ that, $(1-a)\phi^a$ is still a Kantorovich potential from $\mu^a_L$ to $\mu^1_L$.

Now we can apply the following Eulerian calculus for all $\pi_L$, for which the whole evolution is contained in a compact set where in particular all functions involved are bounded with bounded derivatives:
\begin{align}
	\partial_a\int \Delta \phi^a(\eta^a) \d\pi_L(\eta)&=\int  	\partial_a\left( \Delta \phi^a(\eta^a)\right) \d\pi_L(\eta)\\
	&=\int \left(\Delta \frac{|\nabla\phi^a|^2}{2}-\nabla\Delta\phi^a\cdot \nabla\phi^a\right)(\eta^a)\d\pi_L(\eta)\\
	&= \int \left(\Ric_{\m}(\nabla\phi^a)+\|\nabla^2\phi^a\|^2_{\mathrm{HS}}\right)(\eta^a)\d\pi_L(\eta).
\end{align}
We in the above have used \eqref{eq:HJ}, \eqref{eq:direction} and the Bochner identity.
By integration
\begin{align}
	\int \Delta \phi^1(\eta^1) \d\pi_L(\eta)-\int \Delta \phi^0(\eta^0) \d\pi_L(\eta)=\int^1_0\int  \left(\Ric_{\m}(\nabla\phi^a)+\|\nabla^2\phi^a\|^2_{\mathrm{HS}}\right)(\eta^a)\d\pi_L(\eta)\d a.
\end{align}
By the dominated convergence theorem with the $L^1(X,\m)$-integrability we have shown for $ \rho^a\Delta\phi^a$, the left-hand side converges as $L\to \infty$, and so does the right-hand side by the monotone convergence theorem with
\[
\Ric_\m(\nabla\phi^a)\geq K|\nabla\phi^a|^2(\eta^a)\in L^1(\pi).
\]
Together with the linearity of \eqref{ineq:23/03-1} from \eqref{eq:linearityofent}, we conclude that $\Ric_{\m}(\nabla\phi^a)\equiv K|\nabla\phi^a|^2$ and $\nabla^2\phi^a\equiv0 $ on $M$ for all $a\in(0,1)$. \qed

\begin{lemma}\label{lemma:nonsymintegral}
	Let $(M,g,\m)$ be a smooth weighted manifold satisfying $\rcd(K,\infty)$ and \cref{asm:graHK}.
	Then for all $x\in M$ and $t>0$, $|\nabla_xp_t(x,y)|\in{L^1_y(\m)}$.
\end{lemma}
\begin{proof}
	Applying the logarithmic Sobolev inequality \eqref{ineq:LSI} to $f=p_{t,y}\coloneqq p_{t}(\cdot,y)$ gives 
	\begin{equation}\label{eq:logSobolev}
		I_{2K}(t)\frac{|\nabla p_{2t,y}|^2}{p_{2t,y}}(x)\leq P_{t} (p_{t,y}\log p_{t,y})(x)-p_{2t,y}\log p_{2t,y}(x),\quad \forall x\in M.
	\end{equation}
	It suffices to show $P_{t} (p_{t,y}\log p_{t,y})(x)\in L^1_y(\m)$.
	If so then denoting by $F(x,y)$ the right-hand side of \eqref{eq:logSobolev}, which is integrable in $y$ together with \cref{prop:ME}, one has by the Cauchy--Schwarz inequality that
	\begin{align}
		\int_M |\nabla_x p_{2t}(x,y)|\d \m(y)&\lesssim_t \int_M p_{2t}(x,y)^{1/2}f(x,y)^{1/2}\d \m(y)\\
		&\leq \left(\int_M p_{2t}(x,\cdot)\d\m\right)^{1/2}\left(\int_M f(x,\cdot)\d\m\right)^{1/2}<\infty.
	\end{align}
	\smallskip
	
	Case 1: the heat flow is ultracontractive.
	
	Then the heat flow is bounded from $L^2$ to $L^\infty$, yielding all heat kernels to be bounded.
	Thus
	\begin{align}
		P_{t} (p_{t,y}\log p_{t,y})(x)\leq \log \|p_{t,y}\|_{L^{\infty}}\int p_{t,x}(z)p_{t,y}(z)\d \m(z)\lesssim_t p_{2t}(x,y).
	\end{align}
	\smallskip
	Case 2: $\m$ has finite total mass.
	
	By Fubini's theorem, it suffices to show 
	\begin{equation}\label{eq:doubleintegral}
		\int p_{t}(x,z)\int |p_t(y,z)\log p_t(y,z)|\d\m(y)\d \m(z)<\infty.
	\end{equation}
	As $a\log a\geq- e^{-1}$ for all $a$, we have by the log-Harnack inequality \eqref{ineq:LHI} that
	\begin{align}
		\int |(p_t\log p_t)(y,z)|\d\m(y)&\leq \int 2/e+ (p_t\log p_t)(y,z)\d \m(y)\\
		&\leq \log p_{2t}(x,z)+\frac{d^2(x,z)}{4I_{2K}} +\frac{2}{e}\m(X)
	\end{align}
	Further, by the concavity of $a\mapsto \log a$, 
	\[
	\log p_{2t}(x,z)-	\log p_{t}(x,z)\leq \frac{1}{p_t(x,z)}(p_{2t}-p_t)(x,z)
	\]
	and integrating the above over $p_t(x,\d\m)$ obtains
	\begin{equation}
		\int	\log p_{2t}(x,z)p_t(x,\m(\d z))\leq \int \log p_t(x,z)p_t(x,\m(\d z))+\int (p_{2t}-p_t)(x,z)\d\m(z).
	\end{equation}
	Finally, combining above estimates with the moment-entropy estimate for heat kernels from \cref{prop:ME}, the double integrable \eqref{eq:doubleintegral} is finite. This completes the proof.
\end{proof}
\subsection{A characterization of Euclidean spaces}
We have seen in \cref{prop:Eucleadian} that weighted Euclidean spaces with a quadratic weight function $f$ satisfy sharp Wasserstein contraction at every pair of points. 
In this subsection, we show that the converse is also true by iteratively applying the splitting theorem \cref{thm:sharpness}.

We also remark that a manifold with non-negative constant weighted Ricci curvature has to satisfy \cref{asm:graHK}.
Therefore, the assumption \cref{asm:graHK} can be removed when sharp Wasserstein contraction holds for all pairs of points.
%\begin{definition}
%	A weighted complete manifold $(M,g,\m)$ with $\m=e^{-f}\mathrm{\vol}_g$, is called a \emph{gradient Ricci soliton} with potential function $f$, if $\Ric_\m\equiv K$ on $M$ for some $K\in\R$.
%	The Ricci soliton is steady, shrinking or expanding if $K=0$, $K>0$ or $K<0$ respectively.
%\end{definition}
\begin{lemma}\label{lemma:Soliton}
	Let $(M,g,\m)$ be a weighted manifold with $\Ric_\m\equiv K$ for some $K\geq 0$. Then it satisfies \cref{asm:graHK}.
\end{lemma}
\begin{proof}
	The statement is clear for $K>0$ as any $\cd(K,\infty)$ space has finite mass; see \cite[Theorem 4.26]{sturm2006--1}.
	
	The assertion for $K=0$ follows from a classic result on Ricci soliton.
	Indeed, if the weighted Ricci curvature associated with \( \m = e^{-f} \mathrm{vol}_g \) is everywhere $0$, then the triple \( (M, g, f) \), with \( f \) serving as the potential function, is a gradient steady Ricci soliton.  
	Then it is known that the potential $f$ has linear growth i.e. there exists $C>0$ s.t. $|\nabla f|^2\leq C$ on $M$; see e.g. \cite[Section 2]{Munteanu-Wang11}.
	
	In particular, for any $N$ larger than the topological dimension $n$ of $(M,g)$, the $N$-Bakry-\'Emery curvature has lower bound
	\[
	\Ric_N\coloneqq \Ric_\m -\frac{1}{N-n}\nabla f\otimes \nabla f\geq -\frac{C}{N-n}.
	\]
	In other words, the mms $(M,g,\m)$ satisfies $\rcd(K,N)$ with $K=-\frac{C}{N-n}$, and thus the heat flow has to be ultracontractive.
	%Assume that $\m=e^{-f}\mathrm{vol}_g$ and $\Ric_\m\equiv K g$.
\end{proof}

\begin{theorem}\label{thm:rigidity}
	Let $(M,g,\m)$ be a complete weighted Riemannian manifold with $\Ric_\m\geq K$.
	Then $(M,g,\m)$ is isomorphic to the weighted Euclidean space
	\[
	(\R^n,|\cdot|,e^{-f}\mathcal{L}^n),\quad \nabla^2f \equiv K,
	\]
	if and only if for all $x,y\in M$ equality in \eqref{eq:WassersteinK} holds for some $t>0$.
	Moreover, the same conclusion holds for \( K < 0 \) provided that the heat flow on \( (M, g, \m) \) is ultracontractive.
\end{theorem}
\begin{proof}
	When equality in \eqref{eq:WassersteinK} is attained for all $x,y\in M$, the infinitesimal quantity $\theta$ introduced in \cite{sturm2017remarks} 
	\[
	\vartheta^+(x,y)\coloneqq -\liminf_{t\searrow 0}\frac1{t}\log\frac{W_2(p_t(x,\d\m),p_t(y,\d\m))}{d(x,y)}
	\]
	equals to $K $ everywhere on $M$.
	Then it follows from \cite[Theorem 3.1]{sturm2017remarks} that $\Ric_{\m}=\Ric+\nabla^2 f\equiv K$ on $M$ (see also \cite[Remark 4.3]{Erbar-Li-Schultz25} that on smooth weighted manifolds, the equality between the $\vartheta$ quantity and the weighted Ricci curvature does not rely on synthetic finite dimensionality).
	Then \cref{asm:graHK} is satisfied by \cref{lemma:Soliton}.
	Therefore, applying \cref{thm:sharpness} shows that there exists a weighted complete Riemannian manifold $(N,g_N,\n)$ satisfying \eqref{eq:splitting}. 
	
	By the tensorization of heat kernels, the weighted manifold $(N, g_N, \nu)$ also satisfies sharp Wasserstein contraction at each pair of points. Moreover, when $K < 0$, it inherits ultracontractivity as well.
   Therefore, iterating the splitting a finite number of times completes the proof.
\end{proof}

\appendix
\section{Proof of the entropy-moment estimate}\label{appendix}
\begin{proof}[Proof of \cref{prop:ME}]
	The volume growth property for any $\cd(K,\infty)$ space, see e.g. \cite{sturm2006--1}, ensures that there is a constant $\kappa>0$ (possibly depending on $x_0$) s.t. $\m(B(x_0,r))\leq \exp(\frac{\kappa}{2}r^2)$ for all $r>0$. This implies, see e.g. \cite[Section 4]{AGSInvention} that
	\begin{align}\label{ineq:volume}
		\int \exp(-\kappa\cdot d^2(x_0,z))\d \m(z)\leq 1.
	\end{align}
	For any $x\in X$ and $D>2$, by the Cauchy-Schwarz inequality, we have for each $t>0$ 
	\begin{align}
		&	W^4_2(p_t(x,\d \m),\delta_{x_0})=\left(\int d^2(x_0,z)p_t(x,z)\d \m(z)\right)^2\\
		& \leq \int d^4(x_0,z)\exp\left(-\frac{d^2(x_0,z)}{Dt}\right)\d \m(z)\cdot \int p^2_t(x,z)\exp\left(\frac{d^2(x_0,z)}{Dt}\right)\d \m(z).\label{ineq:proof2.6-1}
		%& \lesssim_{D}\int \exp\left(-\frac{d^2(x_0,z)}{2Dt}\right)\d \m(z)\int p_t^2(x,z)\exp\left(\frac{2d^2(x,z)}{Dt}\right)\d \m(z)\cdot \exp\left(\frac{2d^2(x_0,x)}{Dt}\right),
	\end{align}
	Applying \cref{thm:weightedInt} for some $C'>0$ depending only on $D$ and $K$, we have
	\begin{align}
		\int p^2_t(x,z)\exp\left(\frac{d^2(x_0,z)}{Dt}\right)\d \m(z)	&\leq \int p_t^2(x,z)\exp\left(\frac{2d^2(x,z)}{Dt}\right)\d \m(z)\cdot \exp\left(\frac{2d^2(x_0,x)}{Dt}\right)\\
		&\lesssim_{D}\frac{1}{\m(B(x,\sqrt{2t}))}\exp\left(\frac{2d^2(x_0,x)}{Dt}+C' t\right).\label{ineq:proof2.6-2}
	\end{align}
	1) We show first that for any $D>4$, there exist $C,t_0>0$ depending only on $x_0$ and $D$ s.t. for all $t\in(0,t_0]$
	\begin{align}
		W^2_2(p_t(x,\d \m),\delta_{x_0})&\leq C (\m(B(x,\sqrt{2t})))^{-\frac12}\exp\left(\frac{d^2(x_0,x)}{Dt}+Ct\right)\label{ineq:moment}\\
		\|(p_t\log p_t)(x,\cdot)\|_{L^1(X,\m)}&\leq e^{Ct}\left(\frac{1}{\m(B(x,\sqrt{2t}))}+\exp\left(\frac{2d^2(x_0,x)}{Dt}\right)\right)
	\end{align}
	Indeed, for all $t<\frac{1}{2D\kappa}$, there exists a constant $C''$ depending only on $D$ and $\kappa$ s.t.
	\[
	d^4\leq C''\exp\left(\frac{d^2}{2Dt}\right),\quad \forall d>0.
	\]
	Then by the choice of $t$ and \eqref{ineq:volume}
	\[
	\int d^4(x_0,z)\exp\left(-\frac{d^2(x_0,z)}{Dt}\right)\d \m(z)\lesssim_{D,x_0} \int \exp\left(-\frac{d^2(x_0,z)}{2Dt}\right)\d \m(z)\leq 1.
	\]
	Hence with \eqref{ineq:proof2.6-1} \eqref{ineq:proof2.6-2}, \eqref{ineq:moment} follows.
	
	Using the formula for Entropy under the change of reference measure, see e.g. \cite[(2.5)]{Ambrosio-Gigli-Mondino-Rajala}, we have 
	\begin{equation}
		\ent(p_t(x,\cdot))\geq -\kappa \int d^2(x_0,z)p_t(x,z)\d \m(z)=-\kappa W^2_2(p_t(x,\d\m),\delta_{{x_0}}).
	\end{equation}
	On $\{p_t(x,\cdot)\geq 1\}$, it follows again by \cref{thm:weightedInt} that
	\begin{equation}
		\int_{\{p_t(x,\cdot)\geq 1\}} p_{t}(x,z)\log p_t(x,z)\d \m(z)\leq \int p^2_{t}(x,z)\d \m(z)\lesssim_{D}\frac{e^{C't}}{\m(B(x,\sqrt{2t}))}.
	\end{equation}
	Therefore, combining above estimates gives
	\begin{align}
		\|p_t\log p_t(x,\cdot)\|_{L^1}&=2	\int_{\{p_t(x,\cdot)\geq 1\}} (p_{t}\log p_t)(x,\cdot)\d \m-\ent(p_t(x,\cdot))\\
		&\lesssim_{D,x_0}\frac{e^{C't}}{\m(B(x,\sqrt{2t}))}+\exp\left(\frac{2d^2(x_0,x)}{Dt}+C't\right).
	\end{align}
	2) The estimates for general $t$ follows from the regularizing properties of the heat flow being as gradient flows.
	Fix $t_0$ to be $\frac{1}{3D\kappa}$. 
	Applying directly Theorem 4.20 in \cite{AGSInvention} shows for all $t>t_0$ that
	\begin{align}
		W^2_2(p_t(x,\d \m),\delta_{x_0})&\leq \frac{e^{4\kappa (t-t_0)}}{\kappa}\left( \ent(p_{t_0}(x,\cdot))+2\kappa W^2_2(p_{t_0}(x,\d \m),\delta_{x_0})  \right)\label{ineq:19/03-1}\\
		\|(p_t\log p_t)(x,\cdot)\|_{L^1(X,\m)}&\leq\frac{e^{Ct}}{\m(B(x,\sqrt{2t}))}+e^{4\kappa (t-t_0)}\left( \ent(p_{t_0}(x,\cdot))+2\kappa W^2_2(p_{t_0}(x,\d \m),\delta_{x_0})  \right).
	\end{align}
	The estimate for entropy can be produced by the same argument in the previous step.
	The finiteness of the Fisher information follows from \eqref{eq:EVI}.
	
	3) When $(X,d)$ is locally compact, $(X,d)$ is proper by the Hopf-Rinow theorem.
	Thus for any $t>0$ and $U\subset X$ bounded, $\m(B(x,t))$ has a positive lower bound uniformly for all $x\in U$.
	Together with the previous quantitative estimates with upper bounds expressed by the volume of balls and $d(x,x_0)$, we conclude that all the considered functionals of heat kernels are locally bounded.
\end{proof}
%\medskip

%\textbf{Data availability} The manuscript has no associated data.

%\smallskip
%\textbf{Conflict of interest} The authors state that there is no conflict of interest
\bibliographystyle{amsplain}
\bibliography{main.bib}

\providecommand{\bysame}{\leavevmode\hbox to3em{\hrulefill}\thinspace}
\providecommand{\MR}{\relax\ifhmode\unskip\space\fi MR }
\providecommand{\MRhref}[2]{%
  \href{http://www.ams.org/mathscinet-getitem?mr=#1}{#2}
}
\providecommand{\href}[2]{#2}
\begin{thebibliography}{10}

\bibitem{Cavalletti-McCann21}
Afiny Akdemir, Andrew Colinet, Robert~J. McCann, Fabio Cavalletti, and Flavia
  Santarcangelo, \emph{Independence of synthetic curvature dimension conditions
  on transport distance exponent}, Trans. Am. Math. Soc. \textbf{374} (2021),
  no.~8, 5877--5923 (English).

\bibitem{ABS19Gafa}
Luigi Ambrosio, Elia Bru{\'e}, and Daniele Semola, \emph{Rigidity of the
  1-{Bakry}-{\'e}mery inequality and sets of finite {Perimeter} in {RCD}
  spaces}, Geom. Funct. Anal. \textbf{29} (2019), no.~4, 949--1001 (English).

\bibitem{AmbrosioGigli2013}
Luigi Ambrosio and Nicola Gigli, \emph{A user's guide to optimal transport},
  Modelling and optimisation of flows on networks, Lecture Notes in Math., vol.
  2062, Springer, Heidelberg, 2013, pp.~1--155. \MR{3050280}

\bibitem{Ambrosio-Gigli-Mondino-Rajala}
Luigi Ambrosio, Nicola Gigli, Andrea Mondino, and Tapio Rajala,
  \emph{Riemannian {R}icci curvature lower bounds in metric measure spaces with
  {$\sigma$}-finite measure}, Trans. Amer. Math. Soc. \textbf{367} (2015),
  no.~7, 4661--4701. \MR{3335397}

\bibitem{AGS13RMI}
Luigi Ambrosio, Nicola Gigli, and Giuseppe Savar\'{e}, \emph{Density of
  {Lipschitz} functions and equivalence of weak gradients in metric measure
  spaces}, Rev. Mat. Iberoam. \textbf{29} (2013), no.~3, 969--996 (English).

\bibitem{AGSInvention}
\bysame, \emph{Calculus and heat flow in metric measure spaces and applications
  to spaces with {R}icci bounds from below}, Invent. Math. \textbf{195} (2014),
  no.~2, 289--391. \MR{3152751}

\bibitem{AGSRCD2014}
\bysame, \emph{Metric measure spaces with {R}iemannian {R}icci curvature
  bounded from below}, Duke Math. J. \textbf{163} (2014), no.~7, 1405--1490.
  \MR{3205729}

\bibitem{AGS15}
\bysame, \emph{Bakry-{\'e}mery curvature-dimension condition and {Riemannian}
  {Ricci} curvature bounds}, Ann. Probab. \textbf{43} (2015), no.~1, 339--404.

\bibitem{Ambrosio-Trevisan14}
Luigi Ambrosio and Dario Trevisan, \emph{Well-posedness of {Lagrangian} flows
  and continuity equations in metric measure spaces}, Anal. PDE \textbf{7}
  (2014), no.~5, 1179--1234 (English).

\bibitem{Bakry-Gentil-Ledoux14}
Dominique Bakry, Ivan Gentil, and Michel Ledoux, \emph{Analysis and geometry of
  {Markov} diffusion operators}, Grundlehren Math. Wiss., vol. 348, Cham:
  Springer, 2014.

\bibitem{Balogh12HJ}
Zolt{\'a}n~M. Balogh, Alexandre Engulatov, Lars Hunziker, and Outi~Elina
  Maasalo, \emph{Functional inequalities and {Hamilton}-{Jacobi} equations in
  geodesic spaces}, Potential Anal. \textbf{36} (2012), no.~2, 317--337
  (English).

\bibitem{Catino-Mantegazza-Lorenzo}
Giovanni Catino, Carlo Mantegazza, and Lorenzo Mazzieri, \emph{On the global
  structure of conformal gradient solitons with nonnegative {Ricci} tensor},
  Commun. Contemp. Math. \textbf{14} (2012), no.~6, 1250045, 12.

\bibitem{Cavalletti-Huesmann15}
Fabio Cavalletti and Martin Huesmann, \emph{Existence and uniqueness of optimal
  transport maps}, Ann. Inst. Henri Poincar{\'e}, Anal. Non Lin{\'e}aire
  \textbf{32} (2015), no.~6, 1367--1377 (English).

\bibitem{EKS15}
Matthias Erbar, Kazumasa Kuwada, and Karl-Theodor Sturm, \emph{On the
  equivalence of the entropic curvature-dimension condition and {Bochner}'s
  inequality on metric measure spaces}, Invent. Math. \textbf{201} (2015),
  no.~3, 993--1071.

\bibitem{Erbar-Li-Schultz25}
Matthias Erbar, Zhenhao Li, and Timo Schultz, \emph{Synthetic notions of
  {Ricci} flow for metric measure spaces}, Preprint, {arXiv}:2501.07175
  [math.{DG}] (2025), 2025.

\bibitem{Giglisplitting13}
Nicola Gigli, \emph{The splitting theorem in non-smooth context}, Preprint,
  {arXiv}:1302.5555 [math.{MG}] (2013), 2013.

\bibitem{Gigli14splitting}
\bysame, \emph{An overview of the proof of the splitting theorem in spaces with
  non-negative {Ricci} curvature}, Anal. Geom. Metr. Spaces \textbf{2} (2014),
  169--213 (English).

\bibitem{Gigli15MAMS}
\bysame, \emph{On the differential structure of metric measure spaces and
  applications}, Mem. Am. Math. Soc., vol. 1113, Providence, RI: American
  Mathematical Society (AMS), 2015.

\bibitem{Gigli-Pasqualetto-Book}
Nicola Gigli and Enrico Pasqualetto, \emph{Lectures on nonsmooth differential
  geometry}, SISSA Springer Ser., vol.~2, Cham: Springer, 2020.

\bibitem{Grigoryan09HK}
Alexander Grigor'yan, \emph{Heat kernel and analysis on manifolds}, AMS/IP
  Stud. Adv. Math., vol.~47, Providence, RI: American Mathematical Society
  (AMS); Somerville, MA: International Press, 2009.

\bibitem{Han-JMPA21}
Bang-Xian Han, \emph{Rigidity of some functional inequalities on {RCD} spaces},
  J. Math. Pures Appl. (9) \textbf{145} (2021), 163--203 (English).

\bibitem{Kuwadaduality}
Kazumasa Kuwada, \emph{Duality on gradient estimates and {Wasserstein}
  controls}, J. Funct. Anal. \textbf{258} (2010), no.~11, 3758--3774 (English).

\bibitem{LottVillani09}
John Lott and C\'{e}dric Villani, \emph{Ricci curvature for metric-measure
  spaces via optimal transport}, Ann. Math. (2) \textbf{169} (2009), no.~3,
  903--991.

\bibitem{Mccann2001}
Robert~J. McCann, \emph{Polar factorization of maps on {R}iemannian manifolds},
  Geom. Funct. Anal. \textbf{11} (2001), no.~3, 589--608. \MR{1844080}

\bibitem{Munteanu-Wang11}
Ovidiu Munteanu and Jiaping Wang, \emph{Smooth metric measure spaces with
  non-negative curvature}, Commun. Anal. Geom. \textbf{19} (2011), no.~3,
  451--486.

\bibitem{Otto-Villani00}
F.~Otto and C.~Villani, \emph{Generalization of an inequality by {Talagrand}
  and links with the logarithmic {Sobolev} inequality}, J. Funct. Anal.
  \textbf{173} (2000), no.~2, 361--400.

\bibitem{Rajala-Sturm14}
Tapio Rajala and Karl-Theodor Sturm, \emph{Non-branching geodesics and optimal
  maps in strong {$CD(K,\infty)$}-spaces}, Calc. Var. Partial Differential
  Equations \textbf{50} (2014), no.~3-4, 831--846. \MR{3216835}

\bibitem{Savare14}
Giuseppe Savar{\'e}, \emph{Self-improvement of the {Bakry}-{\'e}mery condition
  and {Wasserstein} contraction of the heat flow in {{\(\text{RCD}(K,
  \infty)\)}} metric measure spaces}, Discrete Contin. Dyn. Syst. \textbf{34}
  (2014), no.~4, 1641--1661.

\bibitem{sturm2006--1}
Karl-Theodor Sturm, \emph{On the geometry of metric measure spaces. {I}}, Acta
  Math. \textbf{196} (2006), no.~1, 65--131.

\bibitem{Sturm06-2}
\bysame, \emph{On the geometry of metric measure spaces. {II}}, Acta Math.
  \textbf{196} (2006), no.~1, 133--177.

\bibitem{sturm2017remarks}
\bysame, \emph{Remarks about synthetic upper {R}icci bounds for metric measure
  spaces}, Tohoku Math. J. (2) \textbf{73} (2021), no.~4, 539--564.
  \MR{4355059}

\bibitem{Tamanini-19HK}
Luca Tamanini, \emph{From {Harnack} inequality to heat kernel estimates on
  metric measure spaces and applications}, Preprint, {arXiv}:1907.07163
  [math.{PR}] (2019), 2019.

\bibitem{Villani}
C\'{e}dric Villani, \emph{Optimal transport}, Grundlehren der Mathematischen
  Wissenschaften [Fundamental Principles of Mathematical Sciences], vol. 338,
  Springer-Verlag, Berlin, 2009, Old and new. \MR{2459454}

\end{thebibliography}

\end{document}